\documentclass[11pt,a4paper]{article}

\usepackage[
  top=30mm,
  bottom=30mm,
  left=30mm,
  right=30mm
]{geometry}

\usepackage{url}
\usepackage{setspace}
\usepackage{scrextend}
\usepackage{cite}
\usepackage{multirow}
\usepackage{cancel}
\usepackage{graphicx}

\usepackage{amsmath}
\usepackage{amsthm}
\usepackage{amssymb}
\usepackage{yhmath}
\usepackage{mathtools}
\usepackage{mathrsfs}

\usepackage[all]{xy}
\usepackage[toc,page]{appendix}

\usepackage{titlesec}
\usepackage[nottoc]{tocbibind}
\usepackage[titles]{tocloft}
\usepackage{titletoc}

\usepackage{tikz-cd}
\usepackage{indentfirst}
\usepackage{subcaption}
\usepackage[shortlabels]{enumitem}

\newcounter{mainthm}

\newtheoremstyle{thm_style}
  {7pt plus 2pt minus 1pt} 
  {7pt plus 2pt minus 1pt} 
  {}                       
  {}                       
  {\bfseries}              
  {.}                      
  {0.1em}                  
  {}                       

\theoremstyle{thm_style}

\newtheorem{thm}{Theorem}[section]

\newtheorem{lemma}[thm]{Lemma}
\newtheorem{prop}[thm]{Proposition}

\newtheorem{-thm}[thm]{Definition-Theorem}

\newtheorem{defn-lem}[thm]{Definition-Lemma}

\newtheorem{example}[thm]{Example}

\newtheorem{defn}[thm]{Definition}
\newtheorem{assumption}[thm]{Assumption}

\newtheoremstyle{rmk}
  {7pt plus 2pt minus 1pt} 
  {7pt plus 2pt minus 1pt} 
  {}                       
  {}                       
  {\itshape}               
  {.}                      
  {0.5em}                  
  {}                       

\theoremstyle{rmk}
\newtheorem{rmk}[thm]{Remark}

\newtheoremstyle{note}
  {8pt plus 2pt minus 1pt} 
  {5pt plus 1pt minus 1pt} 
  {\itshape}               
  {10pt}                   
  {\bfseries}              
  {}                       
  {0.5em}                  
  {}                       

\theoremstyle{note}

\newcommand{\id}{\mathrm{id}}

\newcommand{\ev}{\mathrm{ev}}

\newcommand{\e}{\mathbf{e}}

\newcommand{\F}{\mathcal{F}}
\newcommand{\R}{\mathcal{R}}
\newcommand{\C}{C}
\newcommand{\RP}{\mathbb{RP}}
\newcommand{\CP}{\mathbb{CP}}

\newcommand{\g}{\mathfrak{g}}

\allowdisplaybreaks

\numberwithin{equation}{section}

\titlespacing*{\section}
  {0pt}
  {3ex plus 1ex minus 0.2ex}
  {1.25ex plus 0.2ex}

\titleformat{\subsection}[runin]
  {\normalfont\normalsize\bfseries}
  {\thesubsection}
  {0.6em}
  {}
  [.]

\titlespacing*{\subsection}
  {0pt}
  {2.25ex plus 1ex minus 0.2ex}
  {0.7em}

\titleformat{\subsubsection}[runin]
  {\normalfont\normalsize\itshape}
  {\S\,\thesubsubsection}
  {0.6em}
  {}
  [.]

\titlespacing*{\subsubsection}
  {0pt}
  {2ex plus 0.8ex minus 0.2ex}
  {0.7em}

\titleformat{\paragraph}[runin]
  {\small\sffamily\bfseries}
  {}
  {0pt}
  {}

\titlespacing*{\paragraph}
  {0pt}
  {1.5ex plus 0.5ex minus 0.2ex}
  {0.7em}

\setlist[description]{
  font=\normalfont\itshape\textbullet\space,
  topsep=4pt,
  itemsep=2pt,
  parsep=0pt
}

\titlecontents{section}
  [0.72em]
  {\scriptsize\bfseries\vspace{2pt}}
  {\thecontentslabel.\enspace}
  {}
  {\titlerule*[0.38pc]{.}\contentspage}

\titlecontents{subsection}
  [0em]
  {\scriptsize}
  {\qquad\quad\thecontentslabel.\enspace}
  {}
  {\titlerule*[0.5pc]{.}\contentspage}

\titlecontents{subsubsection}
  [0em]
  {\footnotesize\vspace{1pt}}
  {\qquad\qquad\thecontentslabel.\enspace}
  {}
  {\titlerule*[0.5pc]{.}\contentspage}

\usepackage{setspace}
\usepackage{indentfirst}

\AtBeginDocument{%
  \setlength{\parindent}{1.5em}%
  \setlength{\parskip}{0pt plus 0.5pt}%
}

\title{The Relative Fukaya Category as a Closed-Open Deformation}
\author{Mohamad Rabah}
\date{}
\begin{document}

\maketitle
\begin{abstract}
    We give a proof of Conjecture $1.3$ of \cite{strom2024maurer}.
\end{abstract}
\tableofcontents

\begingroup
\renewcommand{\thefootnote}{}
\footnotetext{\hspace{-1.8em}\texttt{Email: mohamad@amss.ac.cn, mohammadbrabah@gmail.com}}
\addtocounter{footnote}{-1}
\endgroup

\section{Introduction}
Let $(M, \omega)$ be a closed monotone symplectic manifold as in Definition \ref{MonotoneDefn} and let $D=D_1\cup\dots\cup D_N$ be an orthogonal symplectic normal divisors of $M$, as in Definition \ref{OrthSimTranDiv}, such that
		\begin{equation}
			2c_1(M)=\sum_{1\leq j\leq N}\lambda_j \operatorname{PD}[D_j]
		\end{equation}
		where $\lambda_j$ are non-negative rational numbers such that $\lambda_j\leq 2$ for all $j=1, \dots, N$. Set $X:=M\setminus D$ and note that by Theorem \ref{McLeanStru}, X admits the structure of a Liouville manifold. For each $j=1, \dots, N$ let $q_j$ be a formal variable of $\deg(q_j)=\lambda_j$. Denote by $\Lambda$ the Novikov ring thought of as the $q$-adic completion of the polynomial ring $\mathbb C[q_1, \dots, q_N]$. 
		\begin{defn}[\cite{strom2024maurer}]
			We denote by $\mathcal{W}(M, D)$ the curved filtered $\Lambda$-linear $A_\infty$-deformation of the wrapped Fukaya category $\mathcal{W}(X)$ as in \cite{borman2024l_}, given by the push-forward of the Maurer-Cartan element $\beta$ as in \cite{strom2024maurer},
		\begin{equation*}
			\operatorname{CO}_*(\beta):=\sum_{d\geq 1}\frac{1}{d!}\operatorname{CO}_d(\beta, \dots, \beta)\in\operatorname{MC}(\mathcal{W}(X)^{\Lambda})
		\end{equation*}
		via the Closed-Open map as in \cite{borman2024l_}. 
		\end{defn}
		In this paper we prove Conjecture $1.3$ of \cite{strom2024maurer}. 
		\begin{thm}
			The full sub-category $\mathcal{W}(M, D)$ whose objects are closed Lagrangian branes in $X$, is filtered quasi-equivalent to $\mathcal{F}(M, D)$, Seidel's relative Fukaya category as in \cite{seidel2002fukaya}. 
		\end{thm}

        \subsection*{Organization and proof strategy}

The main ingredient in the proof is an extended Closed-Open construction. We extend the Closed-Open operations to graded-symmetric maps
\[
\operatorname{CO}^{D}_{d}\colon
\mathfrak g_{>0}^{\otimes d}
\longrightarrow
CC^*\bigl(\mathcal W_{\mathrm{cpt}}(X)\bigr)[2-2d],\] for $d\geq1$; by counting disk-rooted configurations with interior incidence conditions along $D$. These operations restrict to the usual closed-open map on $\mathcal P_{>0}SC^*(X))$, and satisfy the $L_\infty$-morphism identities whenever the inputs lie in $\mathfrak g_{>1}$. The Maurer-Cartan homotopy constructed in \cite{strom2024maurer} lies entirely in the completed strict filtration range $\mathfrak g^\Lambda_{>1}$. Thus, although $\operatorname{CO}^{D}$ is not an $L_\infty$-morphism, its identities suffice to push forward this particular homotopy. The resulting family of curved $A_\infty$-structures connects the relative Fukaya structure to the deformation of the wrapped Fukaya category by the Maurer-Cartan element of \cite{strom2024maurer}. Passing to the associated graded shows that its endpoint evaluations are filtered curved quasi-equivalences in the sense of \cite{perutz2022constructing}, proving the main theorem.

Sections~2-4 recall the geometric and algebraic constructions from \cite{borman2024l_}, \cite{perutz2022constructing}, \cite{sheridan2020versality} and \cite{strom2024maurer}. Section~5 contains the extended Closed-Open map and the proof of the main theorem. Readers familiar with the relevant work as cited above, may proceed directly to Section~5 and consult the preceding sections only for notation and conventions.

    \section*{Acknowledgments} First and foremost, we would like to thank Prof. Sheridan for his communications and for bringing this question to the author’s attention. We also gratefully acknowledge the hospitality of CMS at the University of Cambridge, where the author first learned about this conjecture. In this regard, we thank Prof. Smith for his invitation and Prof. Zhou for financial support. We would also like to thank our teachers, Prof. Fukaya and Prof. McLean, for their continuous support. Moreover, I would like to thank my girlfriend, Tracy Zhang Ke, for her constant encouragement throughout this work. Lastly, we are grateful to the anonymous referee for their constructive remarks.

\section{Preliminaries from Symplectic Geometry}
We start by recalling preliminary definitions that will serve as a setting of our constructions.
	
	\begin{defn}\label{MonotoneDefn}
		Let $(M,\omega)$ be a closed symplectic manifold. $(M, \omega)$ is said to be (positively) monotone, if $[\omega]=2\kappa c_1(M)\in H^2(M; \mathbb{R})$ where $\kappa>0$ is a real number, which we will refer to as the $\textit{monotonicity constant}$. 
	\end{defn}
	\begin{example}
		$\CP^n$ with its a Fubini-Study form $\omega_{FS}$, is a monotone symplectic manifold of monotonicity constant $n+1$. 
	\end{example}
	
	\begin{defn}[Liouville Domains] \label{LiouDomDefn}
		A Liouville domain is a triple $(X, \theta, Z)$ where $(X,\omega=d\theta)$ is a compact exact symplectic manifold with contact boundary $\partial X$ of contact form $\theta|_{\partial X}:=\iota_Z\omega$.\\
		$Z\in TX$ is a globally defined vectorfield on $X$ such that $d\iota_Z\omega=\omega$ and pointing outward along $\partial X$. We call $Z$ the Liouville vectorfield.\\
		If we denote by $\phi^Z_t:\mathbb{R}\times X\rightarrow X$ the flow of $Z$, we define $r:X\rightarrow\mathbb{R}_{\geq 0}$ the $\textit{radial function}$ given uniquely by $\phi^Z_{\log r(x)}(x)\in\partial X$. 
	\end{defn}
    
	\begin{rmk}
		In another words, for $x\in X$, $\log r(x)$ is the time needed so that if we flow $x$ along $Z$, we will reach $\partial X$.\\
	\end{rmk}
    \begin{example}
		Let $(Q, g)$ be any closed Riemannian manifold and $\epsilon>0$ real number. Then, the total space of a cotangent disk bundle $\mathbb{D}_\epsilon Q$ is a Liouville domain of symplectic form given by the restriction of $d\lambda_{can}$ from $T^*Q$ and the Liouville vectorfield is the radial vector along the fibres multiplied by the norm function along the fibre of $\mathbb{D}_\epsilon Q\rightarrow Q$. 
	\end{example}
	
	As $X$ is orientable and using $\phi^Z$ we identify a tubular neighborhood of $\partial X\subset X$ by $(0,1]\times\partial X$. Using such identification we consider the following definition.

	\begin{defn}[Completion]\label{CompletionDefn}
		We define $\hat{X}:=X\cup_\sim (0,\infty)\times \partial X$ where $\sim$ is defined on $(0,1]\times\partial X$ and the identified tubular neighborhood of $\partial X$ as above.\\
		$\hat{X}$ is an open symplectic manifold with a symplectic form $\hat{\omega}:=d\hat{\theta}$ where $\hat{\theta}:=r\theta$. 
	\end{defn}
\begin{example}
		Following the same notation as in the above example, the total space of the cotangent bundle $T^*Q$ is the completion of $\mathbb{D}_\epsilon Q$. 
	\end{example}
\begin{defn}[Contact Shell] \label{ContactShellDefn}
    A contact shell of $\partial X$ is an open neighborhood of $\partial X$ in its completion $\hat{X}$ such that using the Liouville flow as above, such a neighborhood can be identified with $(1-\epsilon,1+\epsilon)\times\partial X$ for some $\epsilon>0$.
\end{defn}
    
	\begin{defn}[Contact-type Almost Complex Structure]\label{Contact-typeAlmostComplex}
	    An almost complex structure $J$ on $\hat{X}$ is said to be of contact-type if $dr=\theta\circ J$. Equivalently, $J(\partial_r)=R_\theta$ where $R_\theta$ is the Reeb vectorfield associated to $\theta$. 
	\end{defn}

     \begin{defn}[Convex Finite-type]\label{ConvexFiniteTypeDefn}
            An exact symplectic manifold $(X,\omega)$ is said to be convex of finite-type if there exists a primitive $\theta\in\Omega^1(X)$ of $\omega$ together with a smooth exhausting function $f:X\rightarrow [0, \infty)$ such that there exists a $c\in[0, \infty)$ where $df(Z)>0$ on $f^{-1}([c, \infty))$ for $Z$ the unique vectorfield given by $\iota_Z\omega=\theta$.  
        \end{defn}

\subsection{Liouville Structure on the Complement of Symplectic Divisor}

The purpose of this section is to set up the setting of our main result and the notations used. We will follow closely \cite{borman2022quantum}, \cite{borman2024l_}, \cite{mclean2012growth} and \cite{ganatra2021log}.
		Let $(M, \omega)$ be a closed symplectic manifold.
		\begin{defn}[Symplectic Divisor]
			A symplectic divisor of $(M, \omega)$ is an embedded, connected, closed submanifold $D\subset M$ of real codimension $2$ such that $\omega|_D$ is a symplectic form on $D$. 
		\end{defn} 
		\begin{example}
			$\CP^n\subset\CP^{n+1}$ is a symplectic divisor of $\CP^{n+1}$ for the Fubini-Study symplectic form. 
		\end{example}
		
			\begin{defn}[Transverse Symplectic Divisors]
			Given a finite collection $\{D_j\}_{1\leq j\leq N}$ of symplectic divisors in $M$, such collection is said to be transverse, if for any subset $I\subseteq\{1, \dots, N\}$
			\begin{equation}\label{orientationExactSequ}
				D_I:=\bigcap_{i\in I}D_i
			\end{equation}
			is a transversely cut-out embedded submanifold of $M$. 
		\end{defn}
		Observe that, for a transverse collection of symplectic divisors $\{D_j\}$ in $M$, we have an exact sequence of vectorbundles over $D_I$
		\begin{equation}\label{orientationExactSeq}
			0\rightarrow TD_I\rightarrow TM|_{D_I}\rightarrow \nu_M D_I\rightarrow 0
		\end{equation}
		where $\nu_M D_I$ is the normal bundle of $D_I\subset M$. Moreover, by the transversality condition on $\{D_j\}$, it follows that we have a splitting
		\begin{equation*}
			\nu_M D_I=\oplus_{i\in I}\nu_M D_i|_{D_I}. 
		\end{equation*}
		Both $TM|_{D_I}$ and $TD_I$ have a canonical choice of orientation given by $\omega|_{D_I}$, yet such choice may not be compatible for different choices of $I\subseteq\{1, \dots, N\}$. To this end and following \cite{mclean2012growth}, we introduce the following definitions.
			\begin{defn}[Simple/Orthogonal Symplectic Divisors]\label{OrthSimTranDiv}
			
			\begin{enumerate}
				\item The transverse collection $\{D_j\}$ of symplectic divisors is said to be $\textit{simple}$, if the orientation on $\nu_M D_I$ given by the symplectic orientation on $TD_i$ under the identification $\nu_M D_I=\oplus_{i\in I}\nu_M D_i|_{D_I}$ is compatible with the exact sequence \ref{orientationExactSeq}, for every subset $I\subseteq\{1, \dots, N\}$. 
				\item The transverse collection $\{D_j\}$ of symplectic divisors is said to be $\textit{orthogonal}$, if for each $i\neq j\in\{1, \dots, N\}$ and for every $x\in D_i\cap D_j$, the $\omega$-complement $T_xD_i^{\perp\omega}\leq TM$ is a subspace of $T_x D_j$. 
			\end{enumerate}
		\end{defn}
			Now suppose that $D_\lambda:=\sum_j\lambda_jD_j$ is an orthogonal simple transverse divisors such that $\lambda_j>0$ and denote by $X:=M\setminus D$. By the Poincare-Lefschetz duality, we have an isomorphism
			\begin{equation*}
				H_2(M,X)\xrightarrow{\cong}\mathbb{Z}^N
			\end{equation*}
			given by $a\mapsto(a.D_j)_{j=1, \dots, N}$. On the other hand, if we denote by $e_j:=(0, \dots,0,1,0,\dots, 0)$ the standard basis of $\mathbb{Z}^N$, the inverse of the above map is given by $e_j\mapsto[u_j]\in H_2(M,X)$ where $u_j:(\mathbb{D}, \partial\mathbb{D})\rightarrow(M,X)$ is a continuous map satisfying:
			\begin{enumerate}
				\item $[u_j].D_i=\delta_{j,i}$ for all $i,j=1, \dots, N$,
				\item $\operatorname{im}(u_j)\cap D_i=\emptyset$ for all $i\neq j$. 
			\end{enumerate}
		By the Universal Coefficient theorem, we have an induced identification of $H^2(M,X;\mathbb{R})\cong H_2(M,X;\mathbb{R})^*\cong\mathbb{R}^N$. We denote by $\operatorname{pd}^{rel}(D_j)\in H^2(M, X; \mathbb{R})$ the image of the dual basis of $e_j\in\mathbb{Z}^N$. Hence, under our assumption, and after choosing $\kappa>0$ integer such that $\kappa\lambda_j\in\mathbb{N}$ for all $j=1, \dots, N$, it follows that
		\begin{equation*}
			[\omega]^{rel}:=\sum_j\kappa_j\operatorname{PD}^{rel}(D_j)
		\end{equation*}
		is a lift of $[\omega]\in H^2(M; \mathbb{R})$ to $H^2(M, X; \mathbb{R})$ where $\kappa_j:=\kappa\lambda_j$ for all $j=1, \dots, N$. 
		Now using the deRham model of $H^2(M,X; \mathbb{R})$, it follows that $\omega^{rel}\in\Omega(M,X)\equiv\operatorname{Cone}(\Omega(M)\rightarrow\Omega(X))[-1]$ and in particular, $\omega^{rel}=(\omega, -\theta)$ where $\theta\in\Omega^1(X)$ such that $d\theta=\omega$. Therefore, using Stokes' Theorem and for a given smooth map $u_j:(\mathbb{D}, \partial\mathbb{D})\rightarrow(M, X)$ as above, it follows that
		\begin{equation*}
			\kappa_j=\int_\mathbb{D}u_j^*\omega-\int_{\partial\mathbb{D}}u_j^*\theta
		\end{equation*}

         \begin{thm}[McLean \cite{mclean2012growth}]\label{McLeanStru}
			$X$ admits a Liouville structure which is convex and of finite-type. 
		\end{thm}
        A proof of this theorem as given in \cite{mclean2012growth} is of two steps. First we find tubular neighborhoods of each of the divisors, using a Moser type argument, having "nice" properties as per \cite{mclean2012growth} and then perform an exact perturbation of the primitive $1$-form to get the desired structure. We illustrate these arguments in the setting when $D$ is smooth. That is, $D\subset M$ is a closed, connected symplectic divisor, such that $[\omega]=\kappa \operatorname{pd}[D]$ in $H^2(M;\mathbb R)$ where $\kappa>0$ is a real number.\\
        Let $[\omega]^{rel}\in H^2(M,X;\mathbb R)$ be a lift of $[\omega]\in H^2(M;\mathbb R)$ as above, and using the deRham model of $H^2(M,X;\mathbb R)$ we fix a lift $\omega^{rel}=(\omega, -\theta)$ of $\omega$, where $d\theta=\omega|_X$. By Lemma $5.9$ of \cite{mclean2012growth}, it follows that there exists a tubular neighborhood $D\subseteq UD$ together with a symplectic fibration $UD\rightarrow D$ such that its fibre is symplectomorphic to the standard disk of area $\pi A^2$ for some real number $A>0$. Using polar coordinates $(r, \phi)$ on the complex plane, we have $\omega_{std}=rdr\wedge d\phi$. Thus we have a symplectic identification of a fibre of $UD\rightarrow D$ with $\{r\leq A\}$. Denote by $\rho:UD\rightarrow [0, A^2/2)$ the smooth radial function given by $\rho\equiv r^2/2$ and note that the Hamiltonian vectorfield of $\rho$ is $X_\rho=\partial_\phi$. In particular, the flow of $X_\rho$ produces a Hamiltonian $S^1$-action on $UD$ which is free on $UD\setminus D$ and fixes each point of $D\subset UD$. Now fix a fibre $F$ of $UD\rightarrow D$ and consider $F\setminus\{0\}$ which is symplectically identified with $D_A\setminus\{0\}:=\{0<r\leq A\}\subset\mathbb C$. Observe that, $\frac{1}{2\pi}(\pi r^2-\kappa)d\phi$ is a primitive of $\omega_{std}$ on $D_A\setminus\{0\}$. By exactness of $\omega|_X$ and as $\theta$ is a primitive of $\omega|_X$, it follows that there exists a smooth function $f$ such that $\theta-\frac{1}{2\pi}(\pi r^2-\kappa)d\phi=df$. Now set $\theta':=\theta-df$ and $Z$ its associated Liouville vectorfield. That is the unique vectorfield such that $\iota_Zd\theta'=\theta'$. Moreover, $\theta'(X_\rho)=\rho-\kappa<0$ for $\pi A^2<\kappa$. Therefore, since $\kappa>0$ and hence for $A$ small enough it follows that $Z$ points outward on $M\setminus UD$. \\
       
        As for the case when $D$ is an orthogonal simple crossing, we recall that we have fixed a lift $[\omega]^{rel}\in H^2(M,X;\mathbb R)$ of $[\omega]$ such that $[\omega]^{rel}=\sum_{j=1}^N\kappa_j\operatorname{PD}^{rel}[D_j]$ where $\operatorname{PD}^{rel}$ denotes the Poincare-Lefsechtz duality and $\kappa_j=\kappa\lambda_j$ for $j=1, \dots, N$. Moreover we have fixed a representative $(\omega, -\theta)$ in $\Omega^2(M,X;\mathbb R)$. Using Lemma $5.14$ of \cite{mclean2012growth}, we can find for each $j=1, \dots, N$ tubular neighborhoods $D_j\subset UD_j$ as in the single component case, together with smooth radial functions $r_j:UD_j\rightarrow [0, A^2/2)$ such that:
        \begin{enumerate}
            \item $UD_i\cap UD_j$ is invariant under the Hamiltonian $S^1$-action given by $X_{r_i}$, 
            \item $r_i$ and $r_j$ Poisson commute on $UD_i\cap UD_j$.
        \end{enumerate}
		Moreover, for any subset $I\subseteq\{1,\dots, N\}$ and if we denote by $D_I:=\bigcap_{i\in I}D_i$, there exists a tubular neighborhood $D_I\subset UD_I$ together with a symplectic fibration $UD_I\rightarrow D_I$ of fibre $F_I=\Pi_{i\in I}D_A$ polydisk each of area $\pi A^2$ with a Hamiltonian $(S^1)^{|I|}$-action generated by $X_{r_i}$ for $i\in I$. Similarly, as above we find an exact perturbation of $\theta$ so that on $\Pi_{i\in I} D_A\setminus\{0\}$, $\theta'=\frac{1}{2\pi}\sum_{i\in I}(\pi r_i^2-\kappa_i)d\phi_i$. The Liouville vectorfield $Z$ given by $\theta'$ satisfies $Z(r_i^2/2)=r_i^2/2-\kappa_i$. Thus, if we choose $0<A<\min_{i=1,\dots, N}\kappa_i$, it follows that $Z$ points outwards along $M\setminus\bigcup_{i=1}^N\{r_i\leq A\}$.\\

        In fact, following \cite{borman2022quantum} and following the notation as above, we can also arrange the above radial functions as follows.
        \begin{thm}[\cite{borman2022quantum}]\label{radialFunctionThm}
            There exists a smooth exhausting function $\rho:X\rightarrow \mathbb R_{\geq0}$ such that:
            \begin{enumerate}
                \item $Z(\rho)=\rho$, and $\rho$ defines a Liouville structure which is convex and of finite-type on $X$ as in Definition \ref{ConvexFiniteTypeDefn},
                \item the completion of the Liouville domain $\overline{X}$ given by $\{\rho\leq 1\}$ is $\textit{Liouville}$ equivalent to $X$ as in Defintion \cite{mclean2012growth},
                \item $\rho$ has a continuous extension on $M$,
                \item if $h:\mathbb R\rightarrow \mathbb R$ is a smooth function such that $h$ is constant in a neighborhood of $0\in\mathbb R$ then, $h\circ\rho:M\rightarrow \mathbb R$ is a smooth function. 
            \end{enumerate}
        \end{thm}
        \begin{rmk}
            As per the first item in the above theorem and following the terminology in \cite{strom2024maurer}, we will call $\rho$ $\textit{radial}$ function. 
        \end{rmk}

        \section{Domain Moduli Spaces}\label{treeChapter}

        Following \cite{borman2024l_}, we introduce the following notations.
			\begin{defn}
			Let $\R_k:=\{(z_0, \dots, z_k)\in\operatorname{Conf}_{k+1}(\CP^1)\}/\operatorname{PSL}_2(\mathbb{C})$ be the moduli space of distint marked points on $\CP^1$. 
		\end{defn}
		\begin{prop}[\cite{seidel2008fukaya}]
			For $k\geq2$, $\R_k$ is a complex analytic manifold of $\dim_\mathbb{C}\R_k=k-2$. 
		\end{prop}
		\begin{defn}
			Let $\R^{\text{disk}}_k:=\{(z_0, \dots, z_k)\in \operatorname{Conf}^{or}_{k+1}(\partial\mathbb{D})\}/\operatorname{PSL}_2(\mathbb{Z})$ be the moduli space of cyclically ordered, with respect to the anti-clockwise orientation on $\partial\mathbb{D}$, $(k+1)$-distinct boundary points on $\mathbb{D}$. 
		\end{defn}
		\begin{prop}[\cite{seidel2008fukaya}]
			For $k\geq 2$, $\R^{\text{disk}}_k$ is a smooth manifold of $\dim_\mathbb{R}\R^{\text{disk}}_k=k-2$. 
		\end{prop}
	Denote by $\overline{\R}_k$ the Deligne-Mumford compactification of $\R_k$ and consider the embedding of $\R^{\text{disk}}_k\hookrightarrow\R_k$ induced by the embedding $\mathbb{D}\hookrightarrow\CP^1$ mapping $\partial\mathbb{D}$ into the equator. The Deligne-Mumford compactification $\overline{\R}^{\text{disk}}_k$ of $\R^{\text{disk}}_k$ is the closure of the image of the above embedding in $\overline{\R}_k$. Moreover, such compactification is stratified by smooth manifolds, modelled over ribbon trees.
	\begin{defn}[$k$-leafed Tree]
		A $k$-leafed tree is a connected, cycle-free, directed, rooted, ribbon graph $T$ with $(k+1)$-semi-infinite edges such that:
		\begin{enumerate}
			\item All its vertices are internal vertices. We denote the set of all vertices of $T$ by $V(T)$. 
			\item For each $v\in V(T)$, there exists a unique edge $e_{in}(v)$ adjacent to $v$ called the $\textit{incoming edge}$ of $v$.
			\item If we denote by $e(v)$ the set of all edges adjacent to $v\in V(T)$, then we require that $e(v)\setminus e_{in}(v)\neq\emptyset$, which we call the $\textit{outgoing edges}$ of $v$. 
			\item If we denote by $E(T)$ the set of all edges of $T$ then, we require that there exists a distinguished $\textit{incoming}$ semi-infinite edge $e\in E(T)$ called the $\textit{root}$. All the other semi-infinite edges of $T$ are called $\textit{leaves}$.  We also denote by $iE(T):=E(T)\setminus\{\text{root, leaves}\}$ the set of all internal edges.
		\end{enumerate}
		We denote by $\mathcal{T}(k)$ the set of all $k$-leafed trees. 
	\end{defn}
	\begin{defn}
		Two $k$-leafed trees $T$ and $T'$ are said to be equivalent if they have isotopic embeddings in $\mathbb{R}^3$. 
	\end{defn}
	\begin{defn}
		We denoted by $\mathcal{T}^{or}(k)\subseteq\mathcal{T}(k)$ the set of all $k$-leafed trees with ribbon structure.
		We use the ribbon structure to label all the semi-infinite edges of a given $T\in\mathcal{T}^{or}(k)$ by $\{0, \dots, k\}$ where the root is labeled with $0$. 
	\end{defn}
	\begin{defn}
		A tree $T\in\mathcal{T}(k)$ is said to be stable if every $v\in V(T)$ has valency $|v|$ at least $3$. In particular, any vertex of a stable tree has at least two outgoing edges. We denote the set of all stable $k$-leafed trees by $\mathcal{T}_{st}(k)$. 
	\end{defn}
	
	\begin{thm}[\cite{borman2024l_}]
		$\overline{\R}^{\text{disk}}_k$ is a compact smooth manifold with boundaries and corners of $\dim_\mathbb{R}\overline{\R}^{\text{disk}}_k=k-2$. Moreover, it admits a stratification
		\begin{equation*}
			\overline{\R}^{\text{disk}}_k=\bigcup_{T\in\mathcal{T}^{or}_{st}(k)}\R^{\text{disk}}_T
		\end{equation*}
		where
		\begin{equation*}
			\R^{\text{disk}}_T:=\Pi_{v\in V(T)}\R^{\text{disk}}_{|v|-1}
		\end{equation*}
		is a smooth manifold of $\dim_\mathbb R	\R^{\text{disk}}_T=k-2-|iE(T)|$.
	\end{thm}
	Following the work of Kontsevich \cite{kontsevich2003deformation}, we will introduce a Fulton-Macpherson type of compactification of $\R^{\text{disk}}_{d, k}$, which we denote by $\overline{\R}^{\text{disk}}_{d, k}$. Such compactification, which is different than the Deligne-Mumford one as in \cite{seidel2008fukaya}, is modelled over $2$-colored trees. 
	\begin{defn}
		Let $\underline{\R}_k$ be a Fulton-Macpherson type of compactification of $\R_k$ given by real-blow up along the normal crossing divisor $\overline{\R}_k\setminus\R_k=\bigcup_{k_1+k_2=k+2}\overline{\R}_{k_1}\times\overline{\R}_{k_2}$.
	\end{defn}
	
	\begin{prop}[\cite{borman2024l_}]
		$\underline{\R}_k$ is an $S^1$-bundle over $\overline{\R}_k$. In particular, $\underline{\R}_k$ is a smooth manifold with boundaries and corners of $\dim_\mathbb R\underline{\R}_k=\dim_\mathbb R\overline{\R}_k +1$. 
	\end{prop}
    We will also need to consider the moduli space of disks with boundary and interior marked points. 
	
	\begin{defn}
		Let $\R^{\text{disk}}_{d, k}$ be the moduli space of smooth disks with $(k+1)$-cyclically ordered boundary marked points and $d$-interior marked points. 
	\end{defn}
	
	\begin{prop}[\cite{borman2024l_}]
		In the stable range, that is for $2d+k\geq 2$, $\R^{\text{disk}}_{d,k}$  is a smooth manifold of $\dim_\mathbb R \R^{\text{disk}}_{d,k}=2d+k-2$. 
	\end{prop}
	
	\begin{proof}
		Realize $\R^{\text{disk}}_{d,k}$ as the fixed of point set of the smooth involution $\iota:\R_{k+2d}\rightarrow \R_{k+2d}$ given by complex conjugation. 
	\end{proof}
	
	\begin{defn}[$2$-colored Tree]
		We denote by $\mathcal{T}(d, k)\subset\mathcal{T}(d+k)$ the moduli space of $2$-colored trees where for $T\in\mathcal{T}(d, k)$, the $\textit{coloring data}$ on $T$ is a partition $E(T)=E^{d}(T)\sqcup E^{s}(T)$ of its edges into dashed or solid edges respectively, such that
		\begin{enumerate}
			\item The collection of all solid edges, together with their vertices, form a subtree which is an element of $\mathcal{T}^{or}(k)$. That is, $T\setminus E^{d}(T)\in\mathcal{T}^{or}(k)$. 
			\item Using the directed convention on $T\in\mathcal T(d+k)$, if an edge is flowing from a dashed edge, then such edge is also dashed. 
		\end{enumerate}
		Two $2$-colored trees are said to be equivanlent, if they have an isotopic embedding in $\mathbb{R}^3$, respecting the coloring/edge stratification.\\
		$T\in\mathcal{T}(d, k)$ is stable, if $2d+k\geq 3$ and we denote the subset of all stable $2$-colored trees by $\mathcal{T}_{st}(d, k)$. 
	\end{defn}
	Note that, for $T\in\mathcal{T}(d, k)$ the directed convention and the coloring data on $T$ induces a stratification of its vertices. Namely, $V(T)=V^{d}(T)\sqcup V^{s}(T)$ where $v\in V^{d}(T)\iff e_{in}(v)\in E^{d}(T)$.  
	\begin{prop}[\cite{merkulov2011operads}]
		$\overline{\R}^{\text{disk}}_{d,k}$ is a compact smooth manifold with boundaries and corners together with a stratification by smooth manifolds
		\begin{equation*}
			\overline{\R}^{\text{disk}}_{d, k}=\bigcup_{T\in\mathcal{T}_{st}(d, k)}\R^\text{disk}_T
		\end{equation*}
		where
		\begin{equation*}
			\R^{\text{disk}}_T:=\Pi_{v\in V^{s}(T)}\R^{\text{disk}}_{k_v, d_v}\times \Pi_{v\in V^{d}(T)}\R^{al}_{d_v}
		\end{equation*}
		where, $d_v$ is the number of outgoing dashed edges adjacent to $v$ and $k_v$ is the number of outgoing solid edges adjacent to $v$. Moreover, $\dim_\mathbb R \R^{\text{disk}}_{T}=2d+k-2-|iE(T)|$.
	\end{prop}
	
	\subsubsection{Extra Decorations}
	\begin{defn}[Framing]
		Let $(C; z_0, \dots, z_d)\in\R_{d}$ be a smooth genus-zero Riemann surface with $(k+1)$-distinct marked points. A framing on $(C; z_0, \dots, z_d)$ at the marked point $z_i$ is a choice of direction $\theta_i\in\RP(T_{z_i}C)$ for every $i=0,\dots, d$.\\
		We say $(C; z_0, \dots, z_d)$ is framed if all its marked points are framed.\\
		A framing $\{\theta_0, \dots, \theta_d\}$ is said to be $\textit{aligned}$, if for every $i=1, \dots, d$, there exists a biholomorphism $\psi_i:C\rightarrow\mathbb P^1$ such that:
		\begin{enumerate}
			\item $\psi_1(z_0)=\infty$ and $\psi_i(z_i)=0$.
			\item $\psi_{i, *}(\theta_0)$ and $\psi_{i, *}(\theta_i)$ point in the positive real-direction.  
		\end{enumerate}
	\end{defn}
	Denote by $\operatorname{Aff}(\mathbb{C},\mathbb{R}_{>0}):=\{z\mapsto az+b:a\in\mathbb{R}_{>0}, b\in\mathbb{C}\}$.
	\begin{prop}[\cite{merkulov2011operads}]
		 The moduli space $\R^{al}_d:=\operatorname{Conf}_d(\mathbb{C})/\operatorname{Aff}(\mathbb{C},\mathbb{R}_{>0})$ of smooth genus-zero curves with $(d+1)$-marked points and aligned framing, has the structure of a smooth manifold of $\dim_\mathbb R\R^{al}_d=2d-3$.
	\end{prop}
Let $F$ be an arbitrary finite set and $p:F\rightarrow\{1, \dots, d\}$ be a map
\begin{defn}[$p$-Flavour/ Sprinkles]
	A $p$-flavour on $(C; z_0, \dots, z_d; \theta_0)\in\R^{al}_d$ is a map $\Psi:F\rightarrow\operatorname{Isom}(C, \mathbb{P}^1)$ such that for every $f\in F$ we have
	\begin{enumerate}
		\item $\Psi(f)(z_0)=\infty$ and $\Psi(f)(z_{p(f)})=0$.
		\item $\Psi(f)_*(\theta_0)$ points along the positive real-direction. 
	\end{enumerate}
	We call such $\Psi$ a $\textit{sprinkle}$ on $(C; z_0, \dots, z_d; \theta_0)$ and $\Psi(f)$ a sprinkle at $p(f)$. 
\end{defn}
Denote by $\R^{al}_{d, p}$ the space of all smooth genus-zero $p$-flavoured framed curves. 
\begin{prop}[\cite{borman2024l_}]
	The forgetful map $\R^{al}_{d, p}\rightarrow \R^{al}_d$ has the structure of an $\mathbb{R}^{|F|}_{>0}$-bundle. In particular, $\R^{al}_{d, p}$ is a smooth manifold of $\dim_\mathbb R \R^{al}_{d, p}= 2d-3+|F|$. 
\end{prop}

In what follows, we need to keep track of the slope of the Hamiltonian corresponding to each marked point. To do so and following \cite{borman2024l_}, we introduce the following definition.

\begin{defn}[Weights]
	Consider an element $C\in \R^{al}_{d, p}$. A weight on $C$ is a tuple $n=(n_0, n_1, \dots, n_d)\in\mathbb{N}^{d+1}$, where each $n_i$ corresponds to the $i^{th}$-marked point such that
	\begin{equation*}
		n_0=n_1+\dots+n_d+|F|.
	\end{equation*}
	We denote by $\R^{al}_{d, p, n}$ the moduli space of smooth genus-zero curves with $p$-flavoring, aligned framings and weight $n$. 
\end{defn}

\begin{prop}
	$\R^{al}_{d, p, n}$ is a smooth manifold of $\dim_\mathbb R \R^{al}_{d, p, n}= 2d-3+|F|$. 
\end{prop}
\begin{proof}
	The forgetful map $\R^{al}_{d, p, n}\rightarrow\R^{al}_{d, p}$ given by forgetting the weight is a finite to one covering map. 
\end{proof}

\begin{defn}[Forgettable Marked Point]\label{forgetfulPoint}
	Let $C\in\R^{al}_{d, p, n}$ and suppose that $z_i$ is a marked point of corresponding weight $n_i$. $z_i$ is said to be a forgettable marked point, if $n_i=0$ and for every $f\in F$, $p(f)\neq i$.\\
     Denote by $Z(n):=\{0\leq i\leq k: z_i\text{ is forgettable}\}$ and by $Z^c(n)$ its complement in $\{0, \dots, d\}$.
\end{defn}
	\begin{rmk}
	In another words, forgettable marked points are those marked points that are not asymptotic to a Hamiltonian orbit and hence we can apply the $\textit{removable of singularity}$ theorem. 
\end{rmk}
Let $(d, p, n)$ be a triple corresponding to the number of marked points, flavour, and weights respectively and consider $\R^{al}_{d, n, p}$. Suppose that, after forgetting all forgetful marked points, $(d,n,p)$ reduces to $(d',n',p')$. We have a forgetful map induced on the Fulton-Macpherson compactification as described above,
    \begin{equation*}
    \overline{\R}^{al}_{d,n,p}\longrightarrow\overline{\R}^{al}_{d',n',p'}
    \end{equation*}
    given by forgetting all the forgettable marked points followed by stabilization. Note that, such map is not well-defined in the case when after forgetting all forgettable marked points, the resulting curve is a (Floer) cylinder. Namely, the above map is not well-defined in the case $d'=1, F=\emptyset, n'\geq1$. \\

\begin{defn}
    Let $(d,p,n)$ be as above. We define $\operatorname{Sym}(d, p):=\{(\sigma_1, \sigma_2)\in\operatorname{Sym}(d)\times\operatorname{Sym}(|F|): p(\sigma_2(f))=\sigma_1(p(f))\}$ the subgroup of the group of symmetry of $(d+|F|)$-elements. $\operatorname{Sym}(d, p)$ acts on $\overline{\R}^{al}_{d,n,p}$ by $(\sigma_1, \sigma_2).(\Sigma, z_0,z_1, \dots, z_d; \psi)=(\Sigma,z_0, z_{\sigma_1(1)}, \dots, z_{\sigma_1(d)}; \psi\circ\sigma_2 )$. 
\end{defn}
	\subsubsection{Asymptotic Ends}
		We denote the:
		\begin{enumerate}
			\item $\textit{half-cylinder}$ by  $Z_{\pm}:=\{(s, t)\in\mathbb{R}\times S^1:\pm s\geq 0\}$. 
			\item $\textit{half-strip}$ by $Z_{\pm}^{st}:=\{(s, t)\in \mathbb{R}\times[0, 1]: \pm s\geq 0\}$. 
		\end{enumerate}
	Let $C$ be a smooth genus-zero Riemann surface, possibly with boundary, and $z\in C$. 
	\begin{defn}[Cylinderical End]
		Suppose that $z\in C$ is an interior point. A cylinderical end at $z\in C$ is a choice of a holomorphic embedding $\epsilon_\pm: Z_\pm\hookrightarrow C\setminus\{z\}$ such that $\lim_{s\rightarrow\pm\infty}\epsilon_\pm(s, t)=z$ and
		\begin{enumerate}
			\item if $C$ has no boundary and we identify $C\cong\mathbb{P}^1$ sending $z\mapsto 0, \infty$ in the case of $\epsilon_+, \epsilon_-$ respectively then, $\epsilon_\pm$ factors through $(s, t)\mapsto e^{-2\pi(s+it)}$;
			\item if $C$ has boundary, and we identify $C\cong\mathbb{D}$ sending $z\mapsto 0$ then, $\epsilon_\pm$ factors through $(s,t)\mapsto e^{-2\pi(s+it)}$.
		\end{enumerate}
	\end{defn}
	
	\begin{defn}[Strip-like End]
		Suppose that $C$ has boundary and $z\in\partial C$ is a boundary point. A strip-like end at $z$ is a choice of holomorphic embedding $\epsilon_\pm:Z^{st}_\pm\hookrightarrow C\setminus\{z\}$ such that $\lim_{s\rightarrow\pm\infty}\epsilon_\pm(s,t)=z$ and if we identify $C\cong\mathbb{D}$ sending $z\mapsto\pm 1$ corresponding to $\epsilon_+, \epsilon_-$ respectively then, $\epsilon_\pm$ factors through $(s, t)\mapsto \frac{\e^{\pi(s+it)}-i}{\e^{\pi(s+it)}+i}$. 
	\end{defn}
    \begin{defn}[Universal Choice of Ends]
    Suppose that $\overline{\R}$ is any of the above the compactified domain moduli spaces. 
        \begin{enumerate}
            \item A universal choice of strip-like end is a fibre-wise smooth map $\epsilon:\overline{\R}\times Z^{st}_\pm\hookrightarrow\overline{\R}^{univ}$ such that, for any $r\in\overline{R}$, $\epsilon(r, .):Z^{st}_\pm\hookrightarrow\overline{\R}^{univ}_r$ is a holomorphic embedding defining a strip-like end.
            \item Similarly a universal choice of cylindrical ends is a fibre-wise smooth map $\epsilon:\overline{\R}\times Z_\pm\hookrightarrow\overline{\R}^{univ}$ such that, for any $r\in\overline{\R}$, $\epsilon(r, .):Z_\pm\hookrightarrow\overline{\R}^{univ}_r$ is a holomorphic embedding defining a cylinderical end. 
        \end{enumerate}
    \end{defn}
    Given two marked Riemann surfaces $(\C_\pm, z_\pm)$ with a choice of ends $\epsilon_\pm$ of the $\textit{same type}$, that is, either both are cylindrical or strip-like, we can $\textit{glue}$ them together using the choice of ends. Indeed, if $\epsilon_\pm:Z^{st}_\pm\hookrightarrow C_\pm$ we denote the glued surface by $C_-\#_\rho C_+:=C_-\sqcup C_+/\epsilon_+(s+l, t)\sim\epsilon_-(s,t)$ where $l:=-\log(\rho)$ and $0<\rho< 1$ is the $\textit{gluing parameter}$. The case of gluing cylinderical ends in the presence of sprinkles is more delicate and we follow the construction in \cite{abouzaid2010open}. Continuing such construction inductively as in Section 9g of \cite{seidel2008fukaya} or section 2.4 of \cite{abouzaid2010open}, at the level of strata of the compactifications of $\overline{\R}$  as described above, we get an $\overline{\R}$-family of choices of ends. We say such choice is $\textbf{consistent}$ if it defines a universal family, as in the above definition.

\section{Pseudo-holomorphic Curve Theory}

In this section we define the $L_\infty$ and $A_\infty$ structures used in this paper, that one obtains from various moduli spaces of pseudo-holomorphic curves. We start with setting the sign and grading conventions that we will use in this paper. Our convention is as in \cite{abouzaid2010open}, \cite{sheridan2015homological}, \cite{sheridan2020versality}, \cite{seidel2008fukaya}, \cite{borman2024l_}.

\subsection{Grading}
 
 \begin{defn}[Grading Datum]
            A grading datum is an abstract group $G$ together with a pair of group morphisms $\mathbb G:=\{\mathbb Z\rightarrow G\rightarrow\mathbb Z/2\}$, such that their composition $\mathbb Z\rightarrow \mathbb Z/2$ is reduction mod-$2$. 
        \end{defn}
        \begin{defn}[$\mathbb G$-grading]
            Let $\mathbb G$ be a grading datum and $M$ be a module. $M$ is said to be $\mathbb G$-graded if $M$ is $G$-graded. That is, $M=\bigoplus_{g\in G}M_g$ such that for all $g\in G, M_g\leq M$ is a sub-module and for all $m\in M\setminus\{0\}$ there exists a unique $g\in G$ such that $m\in M_g$.\\
            In this case, we say $m\in M$ has degree $d\in\mathbb Z$ if $m\in M_g$ such that $d\mapsto g$ under the morphism $\mathbb Z\rightarrow G$. 
        \end{defn}
        We give the category of $\mathbb G$-graded modules a symmetric monoidal structure given by $M_1\otimes M_2\cong M_2\otimes M_1$ under the identification $m_1\otimes m_2\mapsto(-1)^{\deg(m_1)\deg(m_2)}m_2\otimes m_1$.
        \begin{rmk}
            Observe that a $\mathbb G$-graded module has a canonical $\mathbb Z/2$-grading given by $G\rightarrow\mathbb Z/2$. 
        \end{rmk}
        \begin{defn}[$\mathbb G$-graded Line]
            A $\mathbb G$-graded line is a $1$-dimensional vectorspace over $\mathbb R$ together with a $\mathbb G$-grading. In particular, the $\mathbb G$-grading of a line is concentrated in a single degree.\\
            Given two $\mathbb G$-graded lines $l_1$ and $l_2$, we say they are isomorphic and write $l_1\cong l_2$ if there exists an $\textbf{isomorphism of $\mathbb G$-graded lines}$ between them.\\
            An isomorphism of $\mathbb G$-graded lines is an equivalence class of $\mathbb G$-graded $\mathbb R$-linear isomorphisms, under scaling by positive real number. 
        \end{defn}
        \begin{defn}[$\mathbb G$-graded $\mathbb Z/2$-torsor]
            Let $\{a,b\}$ be a $\mathbb Z/2$-torsor and $\mathbb G$ be a grading datum. A $\mathbb G$-grading on $\{a,b\}$ is a pair $(\{a,b\}, g)$ such that $g\in G$.\\
            In this case, we say the $\mathbb G$-grading on the $\mathbb Z/2$-torsor is concentrated in degree $g\in G$. 
        \end{defn}
        \begin{example}
            The set of all isomorphic $\mathbb G$-graded lines is a groupoid under tensor product. Such groupoid is equivalent to the groupoid of $\mathbb G$-graded $\mathbb Z/2$-torsors. 
        \end{example}
        \begin{defn}[$\mathbb G$-normalization]
            Suppose that $(\{a,b\}, g)$ is a $\mathbb G$-graded $\mathbb Z/2$-torsor. We define its $\mathbb G$-normalization to be the $\mathbb G$-graded free group of rank $1$ given by
            \begin{equation*}
                |(\{a,b\}, g)|_\mathbb G:=\mathbb Z\langle a,b\rangle/(a+b)[-g],
            \end{equation*}
            where the notation above means that the $\mathbb{G}$-grading on $|(\{a,b\}, g)|_\mathbb G$ is concentrated in degree $g\in G$. 
        \end{defn}
        Now let $(X,\omega)$ be a symplectic manifold and for $x\in X$ denote by $\operatorname{Lag}_x(X)$ the smooth manifold of all linear Lagrangians subspaces of $(T_xX, \omega_x)$. For instance, if $2n$ is the real dimension of $X$ then $\operatorname{Lag}_{x}(X)\cong U(n)/O(n)$. Let $\operatorname{Lag}(X)\rightarrow X$ be the fibre bundle of fibre $\operatorname{Lag}_x(X)$ over $x\in X$. Consider the long exact sequence on homotopy groups associated to the fibre bundle $\operatorname{Lag}(X)\rightarrow X$. Namely,
        \begin{equation*}
            \dots\rightarrow\pi_2(X)\rightarrow\pi_1(\operatorname{Lag}_x(X))\rightarrow\pi_1(\operatorname{Lag}(X))\rightarrow\pi_1(X)\rightarrow0.
        \end{equation*}
		As the $\textit{abelianization}$ functor is exact and noting that $\pi_1(U(n)/O(n))\cong\mathbb Z$, it follows that we have an exact sequence
        \begin{equation}
            \mathbb Z\rightarrow H_1(\operatorname{Lag}(X))\rightarrow H_1(X)\rightarrow 0.
        \end{equation}
	Now let $\gamma:=\{(V,v):V\in\operatorname{Lag}_x(X)\text{ for some } x\in X, v\in V\}\rightarrow\operatorname{Lag}(X)$ be the tautological vectorbundle over $\operatorname{Lag}(X)$. By exactness of the above sequence, it follows that $H_1(\operatorname{Lag}(X))\rightarrow H_1(X)$ is surjective. Thus, after pairing elements of $H_1(X)$ with $w_1(\gamma)$, the first Stiefel-Whitney class of $\gamma$, we get a surjective group morphism $H_1(\operatorname{Lag}(X))\rightarrow\mathbb Z/2$. In this setting, we have a canonical choice of $\mathbb G$-grading datum. Namely, $G:=H_1(\operatorname{Lag}(X))$ and $\mathbb G=\{\mathbb Z\rightarrow H_1(\operatorname{Lag}(X))\rightarrow\mathbb Z/2\}$, where the left arrow is the one in the above exact sequence.\\

    Denote by $\widetilde{\operatorname{Lag}}(X)\rightarrow\operatorname{Lag}(X)$ the covering space of deck of transformations $H_1(\operatorname{Lag}(X))$.
    \begin{defn}[Lagrangian Brane] \label{LagBraneDefn}
        A Lagrangian brane in $X$ is a Lagrangian $L\subset X$ together with a choice of:
        \begin{enumerate}
            \item Pin structure on $L$ denoted by $P_L$,
            \item a lift $\tilde{L}$ of the map $x\in L\mapsto (x, T_xL)\in\operatorname{Lag}(X)$ to the covering space $\widetilde{\operatorname{Lag}}(X)\rightarrow\operatorname{Lag}(X)$. 
        \end{enumerate}
    \end{defn}
    \subsubsection{Gradings for Hamiltonian Chords}
Let $H:S^1\times X\rightarrow \mathbb{R}$ be a time-dependent Hamiltonian, possibly autonomous. Denote by $X_H$ its associated Hamiltonian vectorfield. That is,
	\begin{equation*}
		-dH=\iota_{X_H}\omega,
	\end{equation*}
 where $d$ is the exterior derivative on $X$. Let $L_0, L_1$ be two Lagrangians in $X$ each equipped with a Brane structure as in Definition \ref{LagBraneDefn}.
 \begin{defn}[Hamiltonian Chord]\label{HamChordDefn}
     A Hamiltonian chord of $X_H$ between $L_0$ and $L_1$ is a smooth curve $y:[0,1]_t\rightarrow X$ such that
     \begin{equation*}
         \dot{y}(t)=X_{H}(y(t)),\text{ } y(0)\in L_0 \text{ and }y(1)\in L_1.
     \end{equation*}
     We denote the set of Hamiltonian chords of $H$ from $L_0$ to $L_1$ by $\mathcal X(L_0,L_1; H)$. 
 \end{defn}
For $i=0,1$ we denote by $L^\#_i:L_i\rightarrow\widetilde{\operatorname{Lag}}(X)$ the associated lift and let $y^\#:[0,1]\rightarrow\widetilde{\operatorname{Lag}}(X)$ be a lift of $y$ such that $y^\#(i)\in L^\#_i$. Denote by $D_{y^\#}$ the associated real Cauchy-Riemann operator on the upper half-plane and let $k\equiv \operatorname{ind}(D_{y^\#})\in\mathbb Z$ be its index. Let $g\in\ker(H_1(\operatorname{Lag}(X))\rightarrow\mathbb Z/2)$ be the unique deck of transformation the two different lifts of $y$ to each other. 

\begin{defn}[Degree of $y$]
    The degree of the Hamiltonian chord $y$ is given by $|y|\equiv\deg(y):=(k-g)\in H_1(\operatorname{Lag}(X))$ with the understanding that $k$ denotes the image of $\operatorname{ind}(D_{y^\#})\in\mathbb Z\mapsto H_1(\operatorname{Lag}(X))$ in the above short exact sequence. 
\end{defn}

Now consider the $\textit{determinant line bundle}$ $\det(D_{y^\#})$ on the upper half-plane where we give it a $\mathbb G$-grading concentrated at $|y|$. Using the Pin-structure on each of $L_0, L_1$ we define a $\mathbb G$-graded $2$-torsor $Pin_{y^\#}$ given by the set of isomorphism classes of Pin-structures on $\gamma\rightarrow\operatorname{Lag}(X)$ together with an identification of such Pin-structures on $L_0^\#$ and $L_1^\#$ at $y^\#(0)$ and $y^\#(1)$ respectively. We set $o_{y^\#}:=\det(D_{y^\#})\otimes Pin_{y^\#}$. We denote by $y^\#_i$ the two different lifts of $y$ to $\widetilde{\operatorname{Lag}}(X)$ for $i=1,2$.\\ 
Following Proposition $8.1.4$ of \cite{fukaya2010lagrangian} and Proposition $1.4.10$ of \cite{abouzaid2010open}, it follows that, we have a $\textit{canonical}$ identification of $o_{y_1^\#}\cong o_{y_2^\#}$ as $\mathbb G$-graded lines. Hence, $o_{y^\#}$ is independent of the choice of lift of $y$ up to an isomorphism of $\mathbb G$-graded lines.
\begin{defn}[Orientation Line]
    We define the orientation line at the Hamiltonian chord $y$ by $o_y\equiv o_{y^\#}$.\\
    We also denote by $|o_y|_\mathbb G$ its associated $\mathbb G$-normalization. 
\end{defn}

\subsubsection{Grading of Hamiltonian Orbits}
 \begin{defn}[Hamiltonian Orbit]\label{HamOrbDefn}
     A Hamiltonian orbit of $X_H$ is a smooth map $x:S^1_t\rightarrow X$ such that
     \begin{equation*}
         \dot{x}(t)=X_{H}(x(t)). 
     \end{equation*}
     We denote the set of Hamiltonian orbits of $H$ by $\mathcal X(H)$. 
 \end{defn}
 Following \cite{audin2014morse} we consider a smooth path of symplectic matrices $\Phi:[0,1]\rightarrow\operatorname{Sp}(2n)$ such that:
 \begin{enumerate}
     \item $\Phi(0)\equiv\id$,
     \item $\ker(\Phi(1)-\id)=\{0\}$. 
 \end{enumerate}
 We denote by $\mathfrak{sp}(2n):= T_{\id}\operatorname{Sp}(2n)$ its Lie algebra and define $A:[0,1]\rightarrow\mathfrak{sp}(2n)$ given by
 \begin{equation*}
     A(t):=(\frac{d}{dt}\Phi(t)).\Phi(t)^{-1}.
 \end{equation*}
 Possibly after a homotopy, we also assume that $A(0)=A(1)$. \\
 Let $B:Z_-\rightarrow\mathfrak{sp}(2n)$ be a smooth map from the negative half-cylinder such that:
 \begin{enumerate}
     \item $B(s,t)\equiv A(t)$ for $s>>1$,
     \item $B\equiv 0$ in a neighborhood of $s=0$.  
 \end{enumerate}
 \begin{defn}
     For a smooth map $f:S^1\rightarrow\operatorname{Lag}(\mathbb R^{2n})$ and $p>2$, we define $W^{1,p}(Z_-; (\mathbb R^{2n}, f))$ to be the Sobolev space consisting $\zeta\in W^{1,p}(Z_-; \mathbb R^{2n})$ such that $\zeta(0, t)=f(t)$. 
 \end{defn}
 Under the setting of the above Definition and given $\Phi$ as above, we define a real Cauchy-Riemann operator $D_{\Phi, f}:W^{1,p}(Z_-; (\mathbb R^{2n}, f))\rightarrow L^p(\hom ^{0,1}(TZ_-, \mathbb R^{2n }))$ given by
 \begin{equation*}
     D_{\Phi, f}(\zeta):=(\nabla\zeta-(B.\zeta)\otimes dt)^{0,1},
 \end{equation*}
using the standard complex structure on $\mathbb R^{2n}\cong \mathbb C^n$. Using Appendix C of \cite{mcduff2012j}, we have $D_{\Phi, f}$ is a Fredholm operator.

\begin{defn}[Non-degenerate Hamiltonian]
    A Hamiltonian $H:S^1\times X\rightarrow \mathbb R$ is said to be non-degenerate if $\operatorname{graph}(\phi^{X_H}_1)$ the graph of the time-$1$ flow of $X_H$ intersects the diagonal in $X\times X$ transversely.
\end{defn}
Suppose that $H$ is a non-degenerate Hamiltonian and $x:S^1\rightarrow X$ is a Hamiltonian orbit of $H$. Let $\tilde{x}:S^1\rightarrow\operatorname{Lag}(X)$ be a lift of $x$ and fix a unitary trivialization $\Psi:x^*TX\cong\underline{\mathbb C}^n$. As $H$ is non-degenerate, we set $\Phi(t):=\Psi(t)\nabla\phi_{X_H}^t(\Psi(0)^{-1})$ and $f(t):=\Psi(t)\tilde{x}(t)$ and consider the associated Fredholm operator $D_{\Phi, f}:W^{1,p}(Z_-; (\mathbb R^{2n}, f))\rightarrow L^p(\hom ^{0,1}(TZ_-, \mathbb R^{2n }))$. Let $[\tilde{x}]\in H_1(\operatorname{Lag}(X))$ be the unique homology class such that $\operatorname{ind}(D_{\Phi, f})\equiv 0$. 

\begin{defn}[Degree of a Hamiltonian Orbit]
    We define the degree of $x$ by $|x|\equiv\deg(x):=n-[\tilde{x}]\in H_1(\operatorname{Lag}(X))$ with the understanding that $n\in\mathbb Z\mapsto H_1(\operatorname{Lag}(X))$ using the above short exact sequence.\\
    We define the orientation line at $x$ by $o_x:=\det(D_{\Phi, f})$, which is a $\mathbb G$-graded line concentrated at $\deg(x)$.\\
    We denote by $|o_x|_\mathbb G$ its $\mathbb G$-normalization. 
\end{defn}

\subsection{$L_\infty$-Structure}
\subsubsection{Algebraic Preliminaries}
\begin{defn}[Filtered Graded Algebra]
			A filtered graded algebra is a tuple $(\Lambda, \mathcal{F})$ where $\Lambda$ is a $\mathbb{Q}$-graded $\mathbb{C}$-algebra
			\begin{equation*}
				\Lambda=\oplus_{i\in\mathbb{Q}}\Lambda_i\text{ such that }\Lambda_i.\Lambda_j\subseteq \Lambda_{i+j}.
			\end{equation*}
			$\mathcal{F}$ is a (decreasing) filtration on $\Lambda$. That is, for every $n\in\mathbb{Z}$, $\mathcal{F}_{\geq n}\Lambda\subseteq \Lambda$ is a $\mathbb{C}$-subspace satisfying
			\begin{equation*}
				\mathcal{F}_{\geq n_1}\Lambda.	\mathcal{F}_{\geq n_2}\Lambda\subseteq \mathcal{F}_{\geq n_1+n_2}\Lambda.
			\end{equation*}
			We will always assume that the $\mathcal{F}$-filtration is exhaustive and bounded from below. 
		\end{defn}
	Let $\mathfrak g$ be a $\mathbb{C}$-vectorspace and consider the base change $\mathfrak g\otimes_\mathbb{C} \Lambda$. We have an induced filtration on $\mathfrak g\otimes_\mathbb{C} \Lambda$, still denoted by $\mathcal{F}$, given by
	\begin{equation*}
		\mathcal{F}_{\geq n}(\mathfrak g\otimes_\mathbb{C} \Lambda):= \mathfrak g\otimes_\mathbb{C} \mathcal{F}_{\geq n}\Lambda.
	\end{equation*}
	We denote by $\overline{\mathfrak g\otimes_\mathbb{C} \Lambda}$ the completion with respect to the induced filtration.
	\begin{defn}[Degree-wise Tensor Product]
		Suppose that $\mathfrak g$ is a $\mathbb{Q}$-graded $\mathbb C$-vectorspace. We define the degree-wise tensor product by
		\begin{equation*}
			\mathfrak{g}\hat{\otimes}\Lambda:=\oplus_{i\in\mathbb{Q}}(\overline{\mathfrak g\otimes_\mathbb{C} \Lambda})_i.
		\end{equation*}
	\end{defn}
	\begin{rmk}\label{ChoiceofRing}
		We will always work in the case when $\mathfrak g$ is $\mathbb{Z}$-graded and concentrated in degrees above some integer. Moreover, $\Lambda$ is a Novikov type ring of the form $\mathbb C[[q_1, \dots, q_n]]$ for some formal variables of fractional degrees $0\leq\deg(q_j)\leq 2$. Moreover, $\mathcal{F}$ is the Novikov filtration induced by $q_j$. With this in mind, we will abuse notation and write $\mathfrak g^\Lambda$ for $\mathfrak g\hat{\otimes}\Lambda$. 
	\end{rmk}
	Now suppose that $(\mathfrak g, \ell_*)$ be an $L_\infty$-algebra over $\mathbb{C}$. We extend the $L_\infty$-maps $\Lambda$-linearily for which we still denote them by $\ell_*$.
	\begin{defn}[Maurer-Cartan Element]
		We denote by $\operatorname{MC}(\mathfrak g^\Lambda)$ the set of all Maurer-Cartan elements of $(\mathfrak g^\Lambda, \ell_*)$. Namely,
		\begin{equation*}
			\operatorname{MC}(\mathfrak g^\Lambda):=\{\alpha\in\mathcal{F}_{\geq 1}\mathfrak g^\Lambda: \deg(\alpha)=2, \mathcal{C}(\alpha)=0\}.
		\end{equation*}
		Where $\mathcal{C}(\alpha):=\sum_{d\geq 1}\frac{1}{d!}\ell_d(\alpha, \dots, \alpha)$ is the $\textit{curvature}$ term. 
	\end{defn}
	\begin{rmk}
		The curvature equation $\mathcal{C}$ is well-defined and convergent with respect to the topology induced by the filtration $\mathcal{F}$ on $\mathfrak g^\Lambda$. In particular, $\mathcal{C}(\alpha)\in \mathfrak g^\Lambda$. 
	\end{rmk}
	
	\begin{defn}[Deformed $L_\infty$-operations]
		Given $\alpha\in\operatorname{MC}(\mathfrak g^\Lambda)$, we define the $\alpha$-deformed $L_\infty$-operations on $\g^\Lambda$ by
		\begin{equation*}
			\ell^\alpha_k(x_1, \dots, x_k):=\sum_{d\geq 0}\frac{1}{d!}l_{d+k}(\alpha^{\otimes d}; x_1,\dots, x_k).
		\end{equation*}
	\end{defn}
	\begin{prop}
		For any $\alpha\in\operatorname{MC}(\mathfrak g^\Lambda)$, $(\mathfrak g^\Lambda, \ell^\alpha_*)$ is an $L_\infty$-algebra. 
	\end{prop}
	One advantage of working over a field of characteristic zero, is that we have a canonical model for the $\textit{homotopy}$ $L_\infty$-algebra. Indeed, consider the graded commutative dg-algebra $\mathbb{C}[t, dt]$ with generators $t, dt$ such that $\deg t=0$, $\deg dt=1$ and $d(t)=dt$. 
	
	\begin{defn}[Homotopy Model]
		Denote by $\mathfrak g_t:=\mathfrak g\otimes_\mathbb C\mathbb{C}[t, dt]$ and define $\overline{\ell}_d$ by
		\[
		\overline{\ell}_k(x_1\otimes\beta_1, \dots, x_d\otimes\beta_d) = 
		\begin{cases}
			(-1)^\epsilon \ell_k(x_1, \dots, x_d)\otimes (\beta_1\dots\beta_d), & d\geq 2, \\
			\ell_1(x_1)\otimes\beta_1+(-1)^{|x_1|}x_1\otimes d\beta_1,   & d=1.
		\end{cases}
		\]
		where $\epsilon=\sum_{i<j}|x_i||\beta_j|$ and $|.|=\deg(.)$. 
	\end{defn}
	\begin{prop}
		$(\mathfrak g_t, \overline{\ell}_d)$ is an $L_\infty$-algebra. 
	\end{prop}
	\begin{defn}[Evaluation/Restriction Map]
		Given $c\in\mathbb{C}$, we have an evaluation map
		\begin{equation*}
			ev_c:\mathfrak g_t\rightarrow \mathfrak g
		\end{equation*}
		given by $ev_c(t)=c$ and $ev_c(dt)=0$. 
	\end{defn}
	\begin{prop}
	For any $c\in\mathbb{C}$, $ev_c:(\mathfrak g_t, \overline{\ell}_*)\rightarrow (\mathfrak g, \ell_*)$ is an $L_\infty$-morphism inducing a quasi-isomorphism $(\mathfrak g_t, \overline{\ell}_*)\simeq (\mathfrak g, \ell_*)$. 
	\end{prop}
	
	\begin{defn}[Gauge-Equivalence]
		$\alpha_0, \alpha_1\in\operatorname{MC}(\mathfrak g^\Lambda)$ are said to be gauge-equivalent, and in this case we write $\alpha_0\sim\alpha_1$, if there exists $\alpha_t\in\operatorname{MC}(\mathfrak g^\Lambda_t)$ such that $ev_0(\alpha_t)=\alpha_0$ and $ev_1(\alpha_t)=\alpha_1$. 
	\end{defn}
	\begin{prop}
		Suppose that $\alpha_0,\alpha_1\in\operatorname{MC}(\mathfrak g^\Lambda)$ are two gauge-equivalent Maurer-Cartan elements of $(\mathfrak g^\Lambda, \ell_*)$. Then, the deformed $L_\infty$-algebras $(\mathfrak g^\Lambda, \ell^{\alpha_0}_1)$, $(\mathfrak g^\Lambda, \ell^{\alpha_1}_1)$ are quasi-isomorphic. 
	\end{prop}

\subsubsection{Geometric Realization}
We start by setting and fixing our sign conventions. To this end we recall the following notions. Let $(\overline{X}, \theta)$ be a Liouville domain as in Definition \ref{LiouDomDefn} and fix a grading datum $\mathbb G$ as above. Let $H:S^1_t\times\overline{X}\rightarrow\mathbb R$ a non-degenerate time-dependent Hamiltonian and $J$ be an $\omega$-compatible almost complex structure on $\overline{X}$. We denote by $g:=\omega(., J.)$ where $\omega:=d\theta$, the induced metric on $\overline{X}$, from which we get an induced $L^2$-metric on $\mathcal L\overline{X}$, the free loop space of $\overline{X}$.
    \begin{defn}[Action Functional]
        For an element $x\in\mathcal{L}\overline{X}$, we define
        \begin{equation*}
            \mathcal{A}_H(x):=-\int_{S^1}x^*\theta+\int_{S^1}H(x(t))dt
        \end{equation*}
    \end{defn}
	\begin{defn}[Floer Cylinder]
		A Floer cylinder between two distinct Hamiltonian orbits $x_+, x_-\in\mathcal{X}(H)$ is a pseudo-holomorphic map $u:\mathbb{R}_s\times S^1_t\rightarrow Y$ such that:
		\begin{enumerate}
			\item $\partial_s u+J(u(s,t))(\partial_t u-X_H(u(s,t))=0$.
			\item $\int_{\mathbb{R}\times S^1}|\partial_s u|^2ds\wedge dt<\infty$.
			\item $\lim_{s\rightarrow\pm\infty}u(s,t)=x_\pm(t)$. 
		\end{enumerate}
        We denote by $\widetilde{\mathcal M}(x_-,x_+)$ the moduli space of all such Floer cylinders asymptotic to $x_-$ and $x_+$. 
	\end{defn}
    \begin{rmk}
        Observe that, due to the fact that $J$ is $s$-independent, it follows that for distinct choices of $x_-, x_+$, we have a free $\mathbb R$-action of the space of all Floer cylinders connecting given by $s_0.u(s, t)=u(s+s_0, t)$ for every $s_0\in\mathbb R$. 
    \end{rmk}
    It is a standard fact \cite{audin2014morse}, that for a generic time-dependent $J_t$ perturbation of $J$, $\widetilde{\mathcal M}(x_-,x_+)$ is a finite-dimensional smooth manifold of dimension $|x_+|-|x_-|$ and so is $\mathcal{M}(x_-,x_+):=\widetilde{\mathcal M}(x_-,x_+)/\mathbb R$.

    \begin{defn}
    \begin{enumerate}
        \item We define the Hamiltonian Floer complex by $CF^*(H):=\bigoplus_{x\in \mathcal{X}(H)}|o_x|_\mathbb G$.
        \item We define the Hamiltonian Floer differential by $\partial^{HF}x_+:=\sum_{x_-\neq x_+}|\mathcal{M}(x_-, x_+)|.x_-$, where the sum is taken over all $x_-\in\mathcal X(H)$ such that $\dim_\mathbb R\mathcal{M}(x_-, x_+)=0$ and $|\mathcal{M}(x_-, x_+)|$ denotes the signed count following our convention as above and section $6$ of \cite{borman2024l_}. 
    \end{enumerate}
    \end{defn}
    
    Suppose that $H_-, H_+$ are two, possibly time-dependent, Hamiltonians on $\overline{X}$ such that $H_-\leq H_+$ point-wise. Let $H_s$ be an $\mathbb R_s\times S^1$-dependent Hamiltonian such that such that:
    \begin{enumerate}
        \item $H_s\equiv H_-$ for $s<<0$,
        \item $H_s\equiv H_+$ for $s>>0$,
        \item $\partial_s H_s\leq 0$. 
    \end{enumerate}
    \begin{defn}[Continuation Map Cylinder]
        Given $x_-$ and $x_+$ periodic orbits of $H_-$ and $H_+$ respectively, as above. A continuation map cylinder between $x_-$ and $x_+$ and for a domain-dependent $J_{s,t}$ is a pseudo-holomorphic map $u:\mathbb R_s\times S^1_t\rightarrow Y$ such that:
        \begin{enumerate}
            \item $\partial_s u+J_{s,t}(u(s,t))(\partial_t u-X_{H_{s,t}}(u(s,t))=0$.
			\item $\int_{\mathbb{R}\times S^1}|\partial_s u|^2ds\wedge dt<\infty$.
			\item $\lim_{s\rightarrow\pm\infty}u(s,t)=x_\pm(t)$. 
        \end{enumerate}
        We denote by $\mathcal M^{al}_2(x_-,x_+)$ the moduli space of continuation map cylinders asymptotic to $x_-$ and $x_+$. 
    \end{defn}
    Also, it is a standard fact that for a generic $(\mathbb R\times S^1)$-dependent perturbations of $(H, J)$, $\mathcal M^{al}_2(x_-,x_+)$ is a finite-dimensional smooth manifold.
    \begin{defn}[Continuation Map]
        We define $c:CF^*(H_-)\rightarrow CF^*(H_+)$ the continuation cochain map, given by $c(x_-):=\sum_{x_+\in\mathcal{X}(H_+)} |\mathcal M^{al}_2(x_-,x_+)|.x_+$ where the sum is taken over all $x_+\neq x_-$ such that $\dim_\mathbb R \mathcal M^{al}_2(x_-,x_+)=0$. 
    \end{defn}
    In the compact case as above, at the linear level, the continuation cochain maps induces linear isomorphisms between the Hamiltonian Floer homologies for different choices of Hamiltonians. Yet in the open case, for instance in the case of a Liouville manifolds as in Definition \ref{CompletionDefn} or \ref{ConvexFiniteTypeDefn}, Hamiltonian Floer homology, when definied, is dependent on the choice of the Hamiltonian. Moreover, the continuation map, when defined, need not be a linear isomorphism.\\
   
We restrict our attention to the setting as in the Introduction. Namely, let $(M,\omega)$ be a (positively) monotone closed symplectic manifold as in Definition \ref{MonotoneDefn} and let $D:=D_1\cup\dots\cup D_N$ be orthogonal simple transverse symplectic divisors as in Definition \ref{OrthSimTranDiv}. Following \cite{borman2024l_}, \cite{strom2024maurer}, we assume that there exists non-negative rational numbers $\lambda_j\leq2$ such that
\begin{equation*}
    2c_1(M)=\sum_{1\leq j\leq N}\lambda_j\operatorname{PD}[D_j]. 
\end{equation*}
Denote by $X:=M\setminus D$ and note that, by Theorem \ref{McLeanStru}, $X$ admits the structure of a Liouville manifold which is convex and of finite-type as in Definition \ref{ConvexFiniteTypeDefn} given by $\rho$ as in Theorem \ref{radialFunctionThm} and Liouville form $\theta$ as constructed in the above section. In order to have a compactness theorem for the moduli spaces we are interested in, we have to pick our auxiliary data to $\textit{preserve}$ the geometry highlighted in the above sub-section. To be more precise, we consider the following definitions.
        \begin{defn}[Adapted Hamiltonians] \label{AdapHamDefn}
            For $j=1, \dots, N$ we denote by $C^\infty(M;D_j)$ to be the Frechet space of all smooth functions $H:M\rightarrow\mathbb R$ such that for any $x\in D_j$, $X_{H, x}\in T_xD_j$.
            We denote by $C^\infty(M;D):=\bigcap_{i=1}^N C^\infty(M;D_j)$. 
        \end{defn}
        \begin{defn}[Adapted Almost Complex structure] \label{AdapJdefn}
            An $\omega$-campatible almost complex structure $J$ on $M$ is said to be adapted if for any $j=1,\dots, N$ and for each $x\in D_j$ we have $J_x(T_xD_j)=T_xD_j$. We denote by $\mathcal{J}(M,D)$, the Frechet space of all such almost complex structures.
        \end{defn}
        \begin{lemma} \label{graphsubmanifold}
            Suppose that $H$ is an adapted Hamiltonian and $J$ an adpated almost complex structure on $M$. Let $(\Sigma,j)$ be a Riemann surface, possibly with boundary, and $\beta\in\Omega^1(\Sigma)$. Then, for every $j=1,\dots, N$, $\Sigma\times D_j\subset \Sigma\times M$ is an almost complex submanifold for $J_\nu$ as in Lemma \ref{GraphTrick}, where $\nu:=-(X_H\otimes\beta)^{0,1}$.
        \end{lemma}
        \begin{proof}
            This follows from Lemma \ref{GraphTrick} and the fact that if $j=1, \dots, N$, $z\in \Sigma$, $x\in D_j$, $v\in T_z\Sigma$ and $V\in T_xD_j$ then,
            \begin{equation*}
                J_x(X_{H, x}(\beta_z(v)))-X_{H,x}(\beta_z(j_zv))+J_xV\in T_xD_j,
            \end{equation*}
            as $H$ and $J$ are adapted. 
        \end{proof}
        Now fix real numbers $0<\epsilon_-<\epsilon<\epsilon_+<1$ and denote by $X_\epsilon$ the Liouville domain given by $\{\rho\leq\epsilon\}$. We assume that $X_{\epsilon_-}$ is Liouville equivalent to $\overline{X}$ as in Theorem \ref{radialFunctionThm}. Let $h:\mathbb R\rightarrow\mathbb R$ be a smooth function such that:
        \begin{enumerate}
            \item $h', h''\geq 0$ for all $x\in\mathbb R$,
            \item for all $x\geq\epsilon_-$, $h(x)=ax+b$ where $a,b\in\mathbb R$ such that $a>0$ and for every $n\in \mathbb N$, $na\notin\operatorname{Spec}(R_{\theta|_{\partial X_\epsilon}})$,
            \item $h\equiv 0$ for all $x\leq\frac{\epsilon_-}{2}$. 
        \end{enumerate}
		
		\begin{thm}[\cite{strom2024maurer}]\label{BasicCofinalHamThm}
			There exists a cofinal sequence $(H_n)_{n\in\mathbb{N}}$ of non-degenerate admissible $S^1$-dependent admissible Hamiltonian functions on $M$ of the form $H_n=nh(\rho)+K_{n,t}\in C^\infty(S^1\times M, \mathbb R)$ such that:
			\begin{enumerate}
				\item $K_{n,t}\leq 0$ and $K_{n,t}\equiv 0$ on the $\textit{neck-region}$ $\{\epsilon_-\leq\rho\leq\epsilon_+\}$.  
				\item For every $n\in\mathbb{N}$, there are finitely many periodic orbits of $H_n$ that are in $X_\epsilon$. Moreover, such orbits are pair-wise disjoint.
				\item If $\gamma$ is a periodic orbit of $H_n$ for some $n\in\mathbb{N}$ such that $\operatorname{im}(\gamma)\cap X_\epsilon\neq\emptyset$. Then, $\operatorname{im}(\gamma)\subset X_\epsilon$. 
				\item  $\gamma$ is a periodic orbit of $H_n$ for some $n\in\mathbb{N}$ such that $\operatorname{im}(\gamma)\cap X_\epsilon=\emptyset$. Then, $\operatorname{im}(\gamma)\subset D$. 
			\end{enumerate}
		\end{thm}
		We will call Hamiltonians of the form as in the above Theorem, $\textit{basic family}$.\\
		In fact, periodic orbits of such a choice of Hamiltonians that lie in $X_\epsilon$ are distinguished from orbits that intersect (and hence lie in) $D$, using a filtration as in \cite{strom2024maurer}. 
		\begin{thm}[$\mathcal{P}$-filtration]\label{P-filtrationResult}
			Suppose that, $\epsilon_+-\epsilon_->0$ is sufficiently small then, there exists a constant $C=C(a, \epsilon)>2$ where $a$ is the slope of the linear part of $h$ such that:
			\begin{enumerate}
				\item $C$ is homogenous in $a$. That is $C(na, \sigma)=nC(a, \sigma)$, for all $n\in\mathbb{N}$.
				\item Given $n\in\mathbb{N}$ and a periodic orbit $x\in\mathcal{X}(H_n)$, the quantity
				\begin{equation}
					\mathcal{P}(x):= Cn_x+\frac{1}{\kappa}A(x)-|x|
				\end{equation}
				defines a complete and exhaustive filtration on $\oplus_{n\in\mathbb{N}}CF^*(X_\sigma, H_n)$. Where $\kappa$ is the monotonicity coefficient and $n_x=na$ the slope of $H_n$.
				\item For any $n\in\mathbb{N}$ and $x\in\mathcal{X}(H_n)$, $\mathcal{P}(x)<0\iff x$ is a $D$-orbit.  
			\end{enumerate}
		\end{thm}

let $\{H_n\}_{n\in\mathbb{N}}$ be a basic cofinal family of non-degenerate Hamiltonians on $X$ and consider the Floer complex
	\begin{equation*}
		CF^*(X, H_n):=\bigoplus _{x\subset X}|o_x|_\mathbb{G}
	\end{equation*}
	generated by non-degenerate $1$-periodic orbits of $H_n$ lying in $X$, which are contractible in $M$. For any $n\in\mathbb{N}$, we denote by
	\begin{equation*}
		c:CF^*(X, H_n)\rightarrow CF^*(X, H_{n+1})
	\end{equation*}
	a continuation chain map.
	
	\begin{defn}[Chain Model of $SH^*(X)$]
		The cochain complex of the symplectic cohomology of $X$ is given by
		\begin{equation*}
			SC^*(X, \{H_n\}_{n\geq1}):=\bigoplus_{n\geq 1} CF^*(X, H_n)[\tau]
		\end{equation*}
		where $\tau$ is a formal variable of $\deg(\tau)=-1$ such that $\tau^2=0$.
		We define the symplectic cohomology differential by
		\begin{equation*}
			\partial^{SH}(x+ \tau y):=(-1)^{|x|}\partial^{HF}x+(-1)^{|y|}(\tau\partial^{HF}y+c(y)-y).
		\end{equation*}
		We define the symplectic cohomology of $X$ by
		\begin{equation*}
			SH^*(X):=H^*(SC^*(X), \partial^{SH}).
		\end{equation*}
	\end{defn}
	
	\begin{thm}
		$SH^*(X)$ only depends on the deformation type of the Liouville manifold $X$. 
	\end{thm}
    Let $(d, p, n)$ be a triple corresponding to the number of marked points, flavor, and weights respectively. Suppose that $r\in\overline{\R}^{al}_{d,p,n}$ has two smooth components, say $r=(r_-,z_-)\#(r_+, z_+)$ glued along $z_-\sim z_+$ where $(r_\pm, z_\pm)\in\R^{al}_{d_\pm, p_\pm, n_\pm}$. We want to choose our perturbation data $(K_r, J_r)$ on $r$ as follows:
    \begin{enumerate}
        \item If $z_-$ has a non-zero weight, we require $(K_r, J_r)|_{r_\pm}=(K_\pm, J_\pm)$, where $(K_\pm, J_\pm)$ are the previously chosen perturbation data on eact component $r_\pm$, respectively.
        \item If $z_-$ has weight-zero, we require $(K_r, J_r)|_{r_-}=(K_-, J_-)$ while $(K_r, J_r)|_{r_+}=(0, J_{r_+})$.  
    \end{enumerate}
    Now fix a consistent choice of universal cylinderical ends $\epsilon_0,\dots,\epsilon_d:\overline{\R}^{al}_{d, p, n}\times Z_\pm\hookrightarrow\overline{\R}^{al, univ}_{d, p,n}$. Using the fact that $\operatorname{Sym}(d,p)$ acts freely on the main component $\overline{\R}^{al}_{d,p,n}\setminus\partial\overline{\R}^{al}_{d,p,n}$ and following the inductive argument as in Section 9g of \cite{seidel2008fukaya} after adhering to the gluing of domains construction in Section 2.4 of \cite{abouzaid2010open}, one can a fibre-wise smooth map
    \begin{equation}
        (K,J):\overline{\R}^{al}_{d,p,n}\rightarrow\Omega^1_{\overline{\R}^{al, univ}_{d,p,n}/\overline{\R}^{al}_{d,p,n}}(\overline{\R}^{al, univ}_{d,p,n}, C^\infty(M, D))\times C
    ^\infty(\overline{\R}^{al, univ}_{d,p,n}, \mathcal{J}(M,D)),
    \end{equation}
    where $\Omega^1_{\overline{\R}^{al, univ}_{d,p,n}/\overline{\R}^{al}_{d,p,n}}(\overline{\R}^{al, univ}_{d,p,n}, C^\infty(M, D))$ are fibre-wise $1$-forms on $\overline{\R}^{al, univ}_{d,p,n}$ with values in $C^\infty(M, D)$, such that:

    \begin{thm}[\cite{strom2024maurer}]\label{L_inftyPerturbationData}
    \begin{enumerate}
        \item $(K_r, J_r)$ is constant on spheres. That is, $(K_r, J_r)=(0, J_r)$ on fibres of the forgetful map $\overline{\R}^{al}_{d,n,p}\rightarrow\overline{\R}^{al}_{d',n',p'}$,
        \item if $r_i$ is a component of $r\in\overline{\R}^{al}_{d, n, p}$ then $(K_r, J_r)|_{r_i}=(K_{r_i}, J_{r_i})$,
        \item $(K, J)$ is $\operatorname{Sym}(d,p)$-equivariant. 
    \end{enumerate}
        In fact, we may choose $(K, J)$ to be of the following form:\\
        for $r\in \overline{\R}^{al}_{d, n, p}$, $K_r\equiv (h(\rho)+K_r)\otimes\gamma_r$ where $\gamma\in\Omega^1(r)$ satisfying:
        \begin{enumerate}
            \item $d\gamma_r\leq 0$,
            \item $\epsilon_i^*\gamma_r=n_idt$ for $|s|>>1$ and $i\in Z^c(n)$, 
            \item $K_r$ is a domain-dependent Hamiltonian such that $K_r\equiv 0$ on the neck-region $\{\epsilon_-\leq\rho\leq\epsilon_+\}$ and for $i\in Z^c(n)$ we have $n_i(h(\rho)+K_r)$ is a non-degenerate admissible Hamiltonian as in Theorem \ref{BasicCofinalHamThm}.
            \item The curvature of the Hamiltonian perturbation $\R(K_r)$ satisfies $\R(K_r)\leq C(na,\epsilon)d\gamma_r$.
            \item $J_r\in\mathcal{J}(M,D)$ is of contact-type on $\{\epsilon_-\leq\rho\leq 1+\epsilon\}$, that is $d\rho=\theta\circ J_r$ such that $\epsilon^*_iJ_r\rightarrow J_{n_i}$ as $s\rightarrow\pm\infty$ and
            \begin{equation*}
                |\nabla^{k}(\epsilon^*_iJ_r-J_{n_i})|\leq e^{-\delta |s|}
            \end{equation*}
                for some $\delta>0$. 
            
        \end{enumerate}
    \end{thm}
    
     Let $x=(x_0, \dots, x_d)$ be a tuple such that:
	\begin{enumerate}
		\item If $i\in Z(n)$ then, $x_i=x^{D_{j_i}}$ a formal variable corresponding to the $j_i^{th}$-component of $D$. 
		\item If $i\in Z^c(n)$ then, $x_i\in CF^*(X, H_{n_i})$ where $n_i$ is the weight at the $i^{th}$-marked point such that $\mathcal{P}(x_i)>0$. Moreover, we require that if $i\neq j\in Z^c(n)$ then, $x_i\neq x_j$. That is, $x_i$ and $x_j$ are distinct Hamiltonian orbits.
	\end{enumerate}
	
	\begin{defn}[$L_\infty$-Moduli Spaces]\label{defnL_inftyModuli}
		We denote by $\mathcal{M}^{al}_{d, p, n}(x):=\{(r, u)\in\mathcal{M}^{al}_{d, p, n, m}(x'): [u].D=|Z(n)|\}$ where $x'$ is formed from $x$ after removing all elements corresponding to $i\in Z(n)$ and $m$ is a tangency vector condition satisfying $m_{i, k}=\delta_{k, j_i}$. 
	\end{defn}

    The major technical result of \cite{borman2024l_}, \cite{strom2024maurer} is that for a generic choice of perturbation data as in the Theorem \ref{L_inftyPerturbationData}, the $L_\infty$-moduli spaces as in the above definition are smooth finite-dimensional manifolds.

    \begin{thm}[\cite{borman2024l_}, \cite{strom2024maurer}]\label{CompactnessResult}
       Suppose that, $|Z^c(n)|\neq1$ or $x_0\neq x_1$ or $F\neq\emptyset$ or $m_i\geq 1$ then, for a generic choice of perturbation data as in Theorem \ref{L_inftyPerturbationData}, $\mathcal{M}^{al}_{d, p,n, m}(x)$ is an oriented finite-dimensional smooth manifold.\\
       In the case when $|Z^c(n)|=1$, $x_0=x_1$, $F=\emptyset$ and $m_i=0$ then, $\mathcal{M}^{al}_{d, p,n, m}(x)$ consists of solutions of the form $x:S^1\rightarrow M$ such that $x=x_0=x_1$. That is, trivial solutions parameterizing Hamiltonian orbits.\\
       Moreover, if $\dim_\mathbb R\mathcal{M}^{al}_{d, p,n, m}(x)=0$ then, $\mathcal{M}^{al}_{d, p,n, m}(x)$ is a compact oriented manifold.\\
       In the case when $\dim_\mathbb R \mathcal{M}^{al}_{d, p,n, m}(x)=1$ and $m_{i,k}=\delta_{i,k}$, that is as in Definition \ref{defnL_inftyModuli}, and for a sequence $\{u_n\}_{n\geq 1}\subseteq\mathcal{M}^{al}_{d, p,n, m}(x)$ having a stable map limit $u_\infty$, we have that:
       \begin{enumerate}
           \item $u_\infty$ has no spherical components,
           \item $u_\infty$ has no component asymptotic to a $\mathcal{P}$-negative orbit. 
       \end{enumerate}
    \end{thm}
    Now let $x_1, \dots, x_k$ be Hamiltonian orbits $x_i\in\mathcal{X}_{n_i}$ such that $\mathcal{P}(x_i)>0$. That is, they are associated to unforgettable marked points carrying weight $n_i>0$. Let $x_0\in\mathcal{X}_n$ for some $n>0$ such that $\mathcal{P}(x)>0$. For a $\textit{rigid}$ curve $u\in\mathcal{M}^{al}_{d,n,p}(x_0,x_1,\dots,x_d)$ the linearization of the (Hamiltonian perturbed) Cauchy-Riemann operator $D_u$ provides us with a linear map
    \begin{equation}
        |o_u|_\mathbb G:t^{i_1}|o_{x_1}|_\mathbb G\otimes\dots\otimes t^{i_d}|o_{x_d}|_\mathbb G\rightarrow |o_{x_0}|_\mathbb G[3-2d]
    \end{equation}
		where $i_j:=|p^{-1}(\{j\})|$. 
    \begin{defn}
        We set $\tilde{\ell}_d(x_1,\dots,x_d):=\sum |\mathcal{M}^{al}_{d,n,p}(x_0,x_1, \dots,x_d)|x_0$, where the sum is taken over all $p,n, x_0$ such that $\dim_\mathbb R\mathcal{M}^{al}_{d,n,p}(x_0,x_1, \dots,x_d)=0$.\\
    \end{defn}
   $\tilde{\ell}_d$ is a linear operator from $\mathcal{P}_{>0}SC^*(X; \{H_n\}_{n\geq1})^{\otimes d}\rightarrow \bigoplus_{n\geq 1}CF^*(X; H_n)$. Moreover, using Section $6$ of \cite{borman2024l_}, $\tilde{\ell}_d$ has a unique $\tau$-extension making it commute with $\partial_\tau$. We abuse notation and still denote such extension by $\tilde{\ell}_d$.\\
   In the case when an $x_i=x^{D_j}$ we follow the above construction, using rigid curves from the $L_\infty$-moduli space as in Definition \ref{defnL_inftyModuli} with the sign convention that $|o_{x^{D_i}}|_\mathbb G\equiv\mathbb K$ and $|x^{D_i}|=2-\lambda_i$. 

    \begin{defn}[Partial $L_\infty$-operations]\label{PartialLinfty}
       Denote by $\mathfrak g:=SC^*(X_\epsilon; \{H_n\}_{n\geq1})\oplus\bigoplus_{j=1}^{N}\mathbb C[\tau].|o_{x^{D_j}}|$ and for $a\in\mathbb R$ we denote by $\mathfrak g_{>a}:=\mathcal{P}_{>a}(SC^*(X_\epsilon; \{H_n\}_{n\geq1}))\oplus\bigoplus_{j=1}^{N}\mathbb C[\tau].|o_{x^{D_j}}|$. We define $\ell_d:\mathfrak g_{>0}^{\otimes d}\rightarrow\mathfrak g_{>0}$ by
       \[
       \ell_d(x_1,\dots,x_d)=
       \begin{cases}
    \tilde{\ell}_d(x_1,\dots,x_d) & \text{for }d\geq 2,\\
    \tilde{\ell}_1(x_1)+(-1)^{|x_1|}\partial_\tau x_1 & \text{for }d=1,
       \end{cases}
       \]
       with the understanding that if all the inputs of $\ell_d$ are of the form $x^{D_j}$ then we set it to zero. 
    \end{defn}

  \subsubsection{Construction of the Maurer-Cartan Element}
  By construction, we have a non-trivial tautological solution of the curvature equation $\mathcal{C}(\alpha)=0$ given by $\alpha:=\sum_{1\leq j\leq N}x^{D_j}.q_j$. Using the homotopy model $\mathfrak g_t^\Lambda:=\mathfrak g^\Lambda\hat{\otimes}\mathbb C[t, dt]$ as described above, consider the $1^{st}$-order linear ODE given by
  \begin{equation*}
      \frac{d}{dt}\beta(t)=\sum_{d\geq0}\frac{1}{d!}\ell_{d+1}(\beta(t), \dots,\beta(t), t\alpha),\\
      \beta(0)=\alpha.
  \end{equation*}
	Using the Fundamental theorem of ODE's, we solve such differential equation term by term, giving us a unique, convergent in the $q$-adic topology on $\Lambda$, solution $\beta(t)\in\mathfrak g_t^\Lambda$. 
    \begin{prop}[\cite{strom2024maurer}]\label{GaugePathProp}
        The unique solution $\beta(t)\in\mathfrak g^\Lambda_t$ has the form $\beta(t)=\alpha(t)+\gamma(t)dt$ such that:
        \begin{enumerate}
            \item $\alpha(t)\in(\mathfrak g^{\Lambda})^2\hat{\otimes}\mathbb C[t]$,
            \item $\gamma(t)\in(\mathfrak g^\Lambda)\hat{\otimes}\mathbb C[t]$,
            \item $\beta(t)\in\operatorname{MC}(\mathfrak g^\Lambda_t)$ and $\beta(t)\in\mathfrak g_{\geq 1,t}^\Lambda$,
        \end{enumerate}
        where $(\mathfrak g^{\Lambda})^n$ denotes degree $n$ elements. 
    \end{prop}
    We set $\beta\equiv \ev_1(\beta(t))\in \operatorname{MC}(\mathfrak g^\Lambda)$.\\

\subsection{$A_\infty$-Structure}
\subsubsection{Algebraic Preliminaries}
Our notions are inline with \cite{fukaya2025unobstructed} and we follow closely the exposition found in \cite{perutz2022constructing}.
\begin{defn}[Curved filtered $\Lambda$-linear $A_\infty$-algebra]
		Given a $\mathbb G$-graded $\Lambda$-module $A$ together with an exhaustive, complete, decreasing filtration $\F_{\geq *}$. $A$ is said to be a filtered $\Lambda$-linear $A_\infty$-algebra, if there exists a sequence of $\Lambda$-linear maps $\{m_k\}_{k\geq 0}$ of $\deg m_k=2-k$, indexed by the non-negative integers,
		\begin{equation*}
			m_k:A^{\otimes k}\rightarrow A
		\end{equation*}
		such that:
		\begin{enumerate}
			\item $m_k$ satisfies the $A_\infty$-relations. That is,
			\begin{equation}
				\sum_{k_1+k_2=k+1}\sum_{i=1}^{k_1} (-1)^*m_{k_1}(a_1, \dots, a_{i-1}, m_{k_2}(a_i, \dots, a_{i+k_2}),a_{i+k_2+1}, \dots, a_{k_1} ), \forall k\geq 0,
			\end{equation}
            where $*\equiv |a_1|+\dots+|a_{k_1}|-k_1$.
			\item Has curvature. That is, $m_0(1)\in\F_{\geq 1}A$,
			\item Preserves the $\F$-filtration. That is, $m_k(\F_{\geq\lambda_1}A, \dots, \F_{\geq\lambda_k}A)\subseteq\F_{\geq\lambda_1+\dots+\lambda_k}A$.  
		\end{enumerate}
	\end{defn}
	In fact, given a filtered, curved $\Lambda$-linear $A_\infty$-algebra as above, we can associate to it a $(\mathbb G\oplus\mathbb Z)$-graded, uncurved, filtered $\Lambda$-linear $A_\infty$-algebra, denoted by $gr_* A$. Indeed, under the assumption that the decreasing filtration is exhaustive, complete and indexed by the integers. We define the $\textbf{full-associated-graded}$ by $gr_*A:=\bigoplus_{p\in\mathbb Z}\F_{\geq p}A/\F_{\geq p+1}A$. $gr_*A$ is bi-graded where the grading is given from the one on $A$ and a $\mathbb Z$-grading given by the filtration. Under the assumption that the $A_\infty$-operators on $A$ are filtration-preserving, we have an induced $A_\infty$-operators on $gr_*A$ given by
    \begin{equation*}
        gr_*(m_k):gr_{\lambda_1}A\otimes\dots\otimes gr_{\lambda_k}\rightarrow gr_{\lambda_1+\dots+\lambda_k}A.
    \end{equation*}
    As the curvature term $A$ satisfies $m_0(1)\in \F_{\geq1}A$, it follows that $(gr_*A, \{gr_*m_k\}_{k\geq0})$ is an uncurved $(\mathbb G\oplus\mathbb Z)$-graded $A_\infty$-algebra.

	We define the notion of filtered, curved $A_\infty$-categories as follows.
	
	\begin{defn}[Filtered curved $A_\infty$-category]
		A filtered curved $A_\infty$-category is a category $\mathcal{A}$ together with a sequence $\{m_k\}_{k\geq0}$ of $\textit{composition maps}$ such that:
		\begin{enumerate}
			\item For any two objects $L_0, L_1$ of $\mathcal{A}$, $\operatorname{Hom}(L_0, L_1)$ is a $\mathbb G$-graded, filtered $\Lambda$-module.
			\item For any countable collection $L_0, L_1, \dots, L_k, \dots $ in $\operatorname{Ob}(\mathcal{A})$ the composition maps
			\begin{equation}
				m_k:\operatorname{Hom}(L_0, L_1)\hat{\otimes}_\Lambda\dots\hat{\otimes}_\Lambda\operatorname{Hom}(L_{k-1}, L_k)\rightarrow\operatorname{Hom}(L_0, L_k)
			\end{equation}
			satisfies the $A_\infty$-relation as above and preserves the filtration.
			\item If in particular, $L_i=L$ for all $i\geq 0$ then, $(\operatorname{Hom}(L, L), \{m_k\}_{k\geq 0})$ is a curved filtered $A_\infty$-algebra. 
		\end{enumerate}
	\end{defn}

    We aim to define the notion of $\textit{quasi-equivalence}$ of curved, filtered $A_\infty$-category. We do so by explaining this notion first in the case of an uncurved $A_\infty$-category and with no filtration.\\
    $\textbf{Uncurved with no filtration case:}$\\
    Note that, for an uncurved $A_\infty$-algebra $(A, \{m_k\}_{k\geq0})$, $m_1\circ m_1=0$ and hence $(A, m_1)$ is a chain complex. Similary, given an uncurved $A_\infty$-category $\mathcal{A}$, we have an induced linear $\textit{cohomological}$ category given as follows.

    \begin{defn}[Cohomological Category]
        The cohomological category of $\mathcal{A}$ denoted by $H\mathcal A$ is a filtered linear category whose objects are the same as $\mathcal A$, morphisms are given by the $m_1$-cohomology. Namely given $a_0, a_1$ objects in $\mathcal A$ then $\hom^{H\mathcal A}(a_0,a_1):=H^*(\hom^{\mathcal{A}}(a_0, a_1), m_1)$ and compositions is given by $[a_0]\circ[a_1]:=(-1)^{|a_1|}[m_2(a_0, a_1)]$. 
    \end{defn}
   In general, the cohomological category of an uncurved $A_\infty$-category need not have identity morphisms. In the case that it does have an identity morphism for each object, we say $H\mathcal{A}$ is $\textbf{unital}$ and that $\mathcal{A}$ is $\textbf{cohomologically-unital}$.

   \begin{defn}[Quasi-Equivalence in the uncurved case]
       Given two unital, uncurved $A_\infty$-categories and a unital $A_\infty$-functor $F:\mathcal A\rightarrow\mathcal B$ between them. $F$ is said to be a quasi-equivalence if the induced functor $HF:H\mathcal A\rightarrow H\mathcal B$ at the cohomological level is a quasi-isomorphism. 
   \end{defn}
   
   Now assume that $\mathcal{A}$ is a filtered, curved $A_\infty$-category over a filtered ring $\Lambda$. Similary as in the algebra case, we denote by $gr_*\mathcal A$ its associated graded. The induced $A_\infty$-structure operators on $gr_*\mathcal A$ are of degree $(2-k, 0)$, making it an uncurved $(\mathbb G\oplus\mathbb Z)$-graded $A_\infty$-category. We say $\mathcal{A}$ is cohomologically unital if its the cohomological category $Hgr_*\mathcal A$ is unital.

\begin{defn}[Filtered, Curved $A_\infty$-functor]
    Given two filtered, curved cohomologically unital $A_\infty$-categories $\mathcal A,\mathcal B$ and functor $\mathcal{G}:\mathcal A\rightarrow\mathcal B$ between them. $\mathcal G$ is said to be a filtered, curved $A_\infty$-functor if
    \begin{enumerate}
        \item $\mathcal G$ is an $A_\infty$-functor,
        \item $\mathcal G$ is filtration-preserving. That is, for $c\in\F_{\geq\lambda}\hom^{\mathcal A}(a_0, a_1)$, we have $\mathcal{G}(c)\in\F_{\geq\lambda}\hom^{\mathcal B}(\mathcal{G}(a_0), \mathcal{G}(a_1))$,
        \item $\mathcal{G}$ is $\textit{curved}$ with $\mathcal{G}^0\in\F_{\geq 1}$,
        \item the induced uncurved $A_\infty$-functor $gr_*\mathcal G:gr_*\mathcal A\rightarrow gr_*\mathcal B$ is a cohomologically unital.
    \end{enumerate}
\end{defn}
\begin{rmk}
       All our $A_\infty$-categories or algebras will be $\textit{unital}$.
   \end{rmk}
\begin{defn}[Quasi-Equivalence in the uncurved case]\label{EquivalenceFunctorUncurved}
    Suppose that $\mathcal A,\mathcal B$ are two uncurved $(\mathbb G\oplus\mathbb Z)$-graded $A_\infty$-categories and $\mathcal{G}:\mathcal A\rightarrow\mathcal{B}$ is a $(\mathbb G\oplus\mathbb Z)$-graded $A_\infty$-functor between them. $\mathcal{G}$ is said to be a quasi-equivalence if the induced functor on the cohomological level $H\mathcal G:H\mathcal{A}\rightarrow H\mathcal B$ is an equivalence of linear categories.\\
    We say $\mathcal A, \mathcal B$ are quasi equivalent if there exists an uncurved $(\mathbb G\oplus\mathbb{Z})$-graded $A_\infty$-category $\mathcal{C}$ and two quasi-equivalences $\mathcal{C}\rightarrow\mathcal{A}$ and $\mathcal{C}\rightarrow\mathcal{B}$. In this case, we write $\mathcal{A\simeq\mathcal B}$. 
\end{defn}
\begin{defn}[Quasi-Equivalence]
    Suppose that $\mathcal A,\mathcal B$ are two filtered, curved $\mathbb G$-graded $A_\infty$-categories and $\mathcal G:\mathcal A\rightarrow\mathcal B$ is a filtered, curved $A_\infty$-functor between them. $\mathcal G$ is said to be a filtered quasi-equivalence if the induced functor $gr_*\mathcal{G}:gr_*\mathcal{A}\rightarrow gr_*\mathcal B$ is a quasi-equivalence as in Definition \ref{EquivalenceFunctorUncurved}. \\
 We say that $\mathcal{A},\mathcal B$ are filtered curved quasi-equivalences and we write $\mathcal{A\simeq\mathcal B}$ if they are connected by a zig-zag diagram of filtered, cuved quasi-equivalences. 
\end{defn}

\subsubsection{Geometric Realization}
We will be working in the same setting as above. Namely, $X=M\setminus D$ together with a choice of Liouville structure $(\theta, Z)$ and convex finite-type structure given by $\rho$ as in Theorem \ref{radialFunctionThm}. We also fix a grading datum $\mathbb G:=\{\mathbb Z\rightarrow H_1(\operatorname{Lag}(X))\rightarrow\mathbb Z/2\}$ and follow the sign and grading convetion as above, which are inline with \cite{borman2024l_}, \cite{perutz2022constructing}, \cite{sheridan2020versality}, and a $\textit{reference}$ $d\theta$-compatible almost complex structure $J_0$ on $X$. We recall the following definition.

\begin{defn}[Lagrangian Label]
    Let $L_0, \dots, L_k$ be a finite collection of Lagrangians in $X$ each equipped with a brane structure as in Definition \ref{LagBraneDefn}. A Lagrangian label is an ordered tuple $\mathbb L=(L_0, \dots, L_k)$ of such Lagrangian branes in $X$. \\
    Consider the domain moduli space $\R^{\text{disk}, al}_{d,k,p}$ with or without a $p$-flavour. For a Lagrangian label $\mathbb L$ as above, we denote by $\R^{\text{disk}, al}_{d,k,p}(\mathbb L)$ the space $\R^{\text{disk}, al}_{d, k,p}$ with the extra decoration given by cyclically labeling the connected components of the underlying curve of $r\in\R^{\text{disk}, al}_{d,k,p}$ minus its boundary marked points by the ordered elements of $\mathbb L$ accordingly. 
\end{defn}

\paragraph{Compact Fukaya Category of $X$:}
The objects of $\mathcal{F}(X)$ the compact Fukaya category of $X$, are closed, exact Lagrangian branes $L\subset \overline{X}$ where $\overline{X}\equiv\{\rho\leq 1\}$.\\
For $L_0, L_1$ two objects of $\mathcal F(X)$, we choose a $[0,1]_t$-dependent Floer datum $(H_{0,1}, J_{0,1})\in C^\infty([0,1]\times X,\mathbb R)\times C^\infty([0,1]; \mathcal J)$ such that:
\begin{enumerate}
    \item $H_{0,1}\equiv 0$ and $J_{0,1}\equiv J_0$ on a neighborhood of $D$ containing $\partial\overline{X}$,
    \item the image of $L_0$ under the time-$1$ flow of $X_{H_{0,1}}$ intersects $L_1$ transversely.  
\end{enumerate}
Using such Floer datum, we define $\hom^{\mathcal F(X)}(L_0, L_1):=\bigoplus_{y\in\mathcal X(L_0, L_1; H_{0,1})}|o_{y}|_\mathbb G$.\\

As for the $A_\infty$-structure on $\mathcal F(X)$, given a countable collection of objects $L_0,\dots,L_k,\dots$ of $\mathcal{F}(X)$, we define
\begin{equation*}
    m_k^{\text{exact}}:\hom^{\mathcal F(X)}(L_0, L_1)\otimes_\mathbb C\dots\otimes_\mathbb C\hom^{\mathcal F(X)}(L_{k-1}, L_{k})\rightarrow\hom^{\mathcal F(X)}(L_0, L_k)[2-k]
\end{equation*}
as follows:\\
Let $\mathbb L:=(L_0, \dots, L_k)$ be the corresponding Lagrangian label and fix a consistent universal choice of strip-like ends $\epsilon_0,\epsilon_1,\dots,\epsilon_k$ on $\R^{\text{disk}}_k$. For $i=0, \dots, k$, we choose a Floer datum $(H_{i,i+1}, J_{i,i+1})$ satisfying the same conditions as above for each consecutive pair $(L_i, L_{i+1})$ with the understanding that for $i=k+1$, we set $i=0$ and the extra requirement that there are no triple intersections. Using the results of \cite{seidel2008fukaya}, we can find a consistent universal choice of perturbation $(K, J):\R^{\text{disk}, univ}_k(\mathbb L)\rightarrow\Omega^1_{\R^{\text{disk}, univ}_k(\mathbb L)/\R^{\text{disk}}_k(\mathbb L)}(\R^{\text{disk}, univ}_k(\mathbb L), C^\infty(X))\times C^\infty(\R^{\text{disk}, univ}_k(\mathbb L), \mathcal J)$ such that for any $r\in\R^{\text{disk}, univ}_k(\mathbb L), K_r$ is supported in a neighborhood of $D$ containing $\partial\overline{X}$ and $J_r\equiv J_0$ on a such neighborhood satisfying
\begin{equation*}
    (\epsilon_i^*K, \epsilon_i^*J)=(H_{i,i+1}dt, J_0) \text{ for } |s|>>1.
\end{equation*}
Now let $y=(y_0,\dots, y_k)$ be an ordered tuple of Hamiltonian chords where $y_i\in\mathcal X(L_i, L_{i+1};H_{i,i+1})$ and for $r\in\R^{\text{disk}}_k(\mathbb L)$ of underlying marked disk $(\Sigma; z_0, \dots, z_k)$ we denote by $C_r:=\Sigma\setminus\{z_0,\dots,z_k\}$ and consider smooth maps $u:C_r\rightarrow M$ such that:
\begin{enumerate}
    \item $[u]=0\in H_2(M, X)$,
    \item for every $z\in\partial_iC_r$ in the $i^{th}$-connected component of $\partial C_r$ with respect to the anti-clockwise orientation, $u(z)\in L_i$,
    \item $(du-Y)^{0,1}=0$, where $Y\in\hom_\mathbb R(TC_r, TX)$ such that for any $v\in TC_r, X_{K(v)}\equiv Y(v)$ and the $(0,1)$-part is taken with respect to $J_r$.
    \item $\lim_{s\rightarrow-\infty}u(\epsilon_0(s,t))=y_0(t)$ and $\lim_{s\rightarrow\infty}u(\epsilon_i(s,t))=y_i(t)$ for $i=1, \dots, k$.
\end{enumerate}
\begin{rmk}
    We note that, in principle, to define the compact Fukaya category $\F(X)$, one has to consider pseudo-holomorphic curves $u:C_r\rightarrow X$. On the other hand, in the view of Lemma \ref{MaxPrinciple}, pseudo-holomorphic curves $u:C_r\rightarrow M$ such that $[u]=0\in H_2(M,X)$ are entirely contained in $X$.
\end{rmk}
Denote by $\mathcal M(\mathbb L):=\{(r, u)\}$ such that $r\in \R^{\text{disk}}_k(\mathbb L)$ and $u$ is a smooth map as above.  
\begin{prop}[\cite{seidel2008fukaya}]
    For a generic choice of consistent universal perturbation datum, $\mathcal{M}(\mathbb L)$ is an orientable finite-dimensional smooth manifold of $\dim_\mathbb R \mathcal M(\mathbb L)=|y_0|-\sum_{1\leq i\leq k}|y_i|+k-2$.\\
    In the case of $\dim_\mathbb R \mathcal M(\mathbb L)=0$, it is a compact manifold. In the case of $\dim_\mathbb R \mathcal M(\mathbb L)=1$, $\mathcal M(\mathbb L)$ has a compactification $\overline{\mathcal M}(\mathbb L)$ by an oriented topological compact manifold with boundary of the form $\mathcal{M}(\mathbb L_1)\times\mathcal{M}(\mathbb L_2)$ such that
    \begin{enumerate}
        \item $\mathbb L_1, \mathbb L_2\subseteq \mathbb L$,
        \item $\mathbb L_1 \cup \mathbb L_2=\mathbb L$,
        \item $\mathbb L_1 \cap \mathbb L_2= L_i$, that is they only have a single Lagrangian in common,
        \item $|\mathbb L_1|, |\mathbb L_2|\geq 2$,
        \item $\dim_\mathbb R \mathcal M(\mathbb L_1),\dim_\mathbb R \mathcal M(\mathbb L_2)=0$.
    \end{enumerate}
\end{prop}
Given the above Proposition and following the sign convention in \cite{seidel2008fukaya}, \cite{borman2024l_}; for a rigid curve $u\in\mathcal M(\mathbb L)$ asymptotic to $y$, we get a linear isomorphism of $\mathbb G$-graded line
\begin{equation*}
    \mu_u:|o_{y_1}|_\mathbb G\otimes\dots\otimes|o_{y_k}|_\mathbb G\rightarrow|o_{y_0}|_\mathbb G[2-k],
\end{equation*}
defining the structural constant of the $A_\infty$-operator $m_k^{\text{exact}}$. 
\begin{thm}[\cite{seidel2008fukaya}]
    The $A_\infty$-structure on $\mathcal F(X)$ is independent, up to quasi-equivalence of $A_\infty$-categories, of the choice of $J_0$, strip-like ends, Floer datum and perturbation datum. 
\end{thm}

\paragraph{Relative Fukaya Category of $(M, D)$:}

Following \cite{seidel2002fukaya} we define $\mathcal{F}(M,D)$ to be the $A_\infty$-category whose objects are closed exact Lagrangian branes $L\subset \overline X$. That is, the same objects as $\F(X)$.\\

 To define the morphism spaces, we consider our Novikov ring $\Lambda$ as defined above. Now for two objects $L_0$ and $L_1$ in $\F(M, D)$, we choose a time-dependent Floer datum $(H_{0,1}, J_{0,1})\in C^\infty([0,1], \mathcal{H}\times\mathcal{J})$ such that:
        \begin{enumerate}
            \item $X_{H_{0,1}}$ preserves $D_j$ for all $j=1, \dots, N$. That is, $H_{0,1}$ is as in Definition \ref{AdapHamDefn}.
            \item The flow of time-$1$ of $X_{H_{0,1}}$ makes $L_0$ and $L_1$ intersect transversely. 
            \item $J_{0,1}$ is adapted to $D_j$ for all $j=1, \dots, N$. That is, $J_{0,1}$ is as in Definition \ref{AdapJdefn}. In particular, each $D_j$ is an almost complex submanifold of $(M,J_{0,1})$. 
        \end{enumerate}
        \begin{defn}
            Under the assumption that $H_{0,1}$ is non-degenerate, we define the $\hom$-space of $L_0$ and $L_1$ in the relative Fukaya category $\F(M, D)$ to be $\hom^{\F(M,D)}(L_0, L_1):=(\bigoplus_{y\in\mathcal{X}(L_0, L_1;H_{0,1})}|o_y|_\mathbb G)\hat{\otimes}_\mathbb C\Lambda$. 
        \end{defn}
        In order to define the $A_\infty$-structure on $\F(M,D)$ as per \cite{perutz2022constructing}, \cite{sheridan2020versality}, we impose the following extra assumption. 
\begin{assumption}\label{Assumption}
    \begin{enumerate}
        \item The $\mathbb Q$-span of $\operatorname{PD}^{rel}[D_j]$ is $H^2(M,X;\mathbb Q)$.
        \item For every $j=1, \dots, N$ the set of $i=1,\dots, N$ such that $\operatorname{PD}[D_i]=\operatorname{PD}[D_j]$ is at least $n+1$ where $2n=\dim_\mathbb R M$. 
    \end{enumerate}
\end{assumption}
    Now let $\mathbb L=(L_0,\dots, L_k)$ be a Lagrangian label and choose Floer datum $(H_{i,i+1}, J_{i,i+1})$ adpated to $D$ with the extra condition that $J_{i, i+1}$ is of contact-type on the neck-region $\{\epsilon_-\leq\rho\leq\epsilon_+\}$. We also require that after taking images by the corresponding Hamiltonian flows, we do not have triple intersections. Let $y=(y_0, y_1, \dots, y_k)$ be a tuple of Hamiltonian chords and for $r\in\R^{\text{disk}}_{d,k}$ given by $r=[\Sigma; z_0,\dots,z_k;\zeta_1, \dots, \zeta_d]$, we denote by $C_r:=\Sigma\setminus\{z_0, \dots,z_k\}$. Fix a consistent universal choice of stip-like ends $\epsilon_0, \epsilon_1,\dots, \epsilon_k$ and a consistent universal choice of perturbations $(K, J)$ such that for every $i=0,\dots, k$, $(\epsilon_i^*K, \epsilon_i^*J)=(H_{i, i+1}, J_{i,i+1})$ for $|s|>>1$. For $m\geq0$ integer, we consider smooth maps $u:C_r\rightarrow M$ satisfying:
    \begin{enumerate}
        \item $(du-Y)^{0,1}=0$, 
        \item for every $z\in\partial_i C_r$, $u(z)\in L_i$,
        \item $\lim_{s\rightarrow-\infty}u(\epsilon_0(s,t))=y_0(t)$ and $\lim_{s\rightarrow\infty}u(\epsilon_i(s,t))=y_i(t)$ for $i=1, \dots, k$.
        \item $\int_{C_r}|du-Y|^2<\infty$,
        \item $[u].D=m$ and $u^{-1}(D)=\{\zeta_1, \dots, \zeta_d\}$. 
    \end{enumerate}

We denote by $\mathcal{M}_{d, k,m}(y):=\{(r,u)\}$ the moduli space of all pairs such that $r\in\R^{\text{disk}}_{d,k}(\mathbb L)$ and $u$ is a smooth map as above.  

\begin{lemma}[Integrand Maximum Principle \cite{perutz2022constructing}]\label{MaxPrinciple}
			Suppose that $u:(\Sigma, \partial\Sigma)\rightarrow M$ is a $J$-holomorphic map such that $u(\partial\Sigma)\subseteq X$. Furthermore, suppose that $[u]=0\in H_2(M, X; \mathbb{Z})$. Then, $\operatorname{im}(u)\subseteq X$. 
		\end{lemma}
		\begin{proof}
			By McLean's Theorem \ref{McLeanStru}, there exists an $f:X\rightarrow\mathbb R$ a smooth exhausting function and a $c>0$ such that $f^{-1}([c, \infty))=X$. Moreover and following the construction in the above Section, there exists a $Z$ a Liouville $1$-form $\theta$ on $X$. WLOG we assume that $u(\Sigma)$ intersects transversely $f^{-1}(\{c\})$. Otherwise, we consider $c_k>c$ a sequence converging to $c$ such that $u(\Sigma)$ intersects $f^{-1}(\{c_k\})$ transversely and argue as follows.\\
            Let $v\in T\partial\Sigma$ be positively-oriented vectorfield and hence, $j(v)$ points inward to $\partial\Sigma$. As $J$ is of contact-type, it follows that, $\theta(du(v))=df\circ J(du(v))$. As $u$ is $J$-holomorphic, we have $df\circ J(du(v))=-df\circ du(j(v))\leq 0$. Under the assumption that $[u]=0\in H_2(M,X;\mathbb Z)$, $u(\partial\Sigma)\subseteq X$ and using Stoke's Theorem, it follows that the energy of the curve lying in $M\setminus X$ is $\theta(du(v))\leq 0$. As $J$ is $\omega$-compatible, it follows that the energy of the curve lying in $M\setminus X$ is zero and hence $u$ is constant on $M\setminus X$. Under the assumption that $u(\Sigma)$ intersects $f^{-1}(\{c\})$ transversely, it follows that $u(\Sigma)\cap (M\setminus X)=\emptyset$ and hence the desired result. 
		\end{proof}
		
		\begin{lemma}[Positivity of Intersection \cite{perutz2022constructing}]\label{PositivityIntersection}
			Suppose that $L\subset X$ is a closed Lagrangian manifold, which is monotone in $M$, and $u:(\Sigma, \partial\Sigma)\rightarrow(M, L)$ is a non-constant pseudo-holomorphic curve. Then, the algebraic intersection $[u].D>0$ and hence the image of $u$ intersects at least one of the components of $D$. 
		\end{lemma}
		\begin{proof}
			The key point is that $H^2(M, X; \mathbb{R})$ is generated by $\operatorname{PD}(D_i)$ which follows from the Poincare'-Lefschetz duality and the assumption on $D_\lambda$. Now suppose for a contradiction that $[u].D_j=0$ for all $j=1,\dots, N$, as the case of $[u].D_j<0$ is ruled out by the fact that $J$ preserves $TD_j$. Then, $[u]=0\in H_2(M, X;\mathbb{Z})$. By Lemma \ref{MaxPrinciple} the $\textit{Integrand maximum principle}$, it follows that $\operatorname{im}(u)\subseteq X$. As $L\subset X$ is necessarily an exact Lagrangian, it follows by Stoke's Theorem that $u$ is a constant map, which is absurd. 
		\end{proof}
The key technical difficulty in such setting is that sphere bubbles with tangency conditions on $D$ are generally a codimension $1$ strata. To this end, we adhere to Sections $5.5, 5.7, 5.8$ of \cite{perutz2022constructing} for our generic choice of perturbation data.

\begin{thm}[\cite{perutz2022constructing}]
    There exists a generic consistent universal choice of perturbation datum so that, $\mathcal{M}_{d, k, d}(y)$ has the structure of an oriented finite-dimensional smooth manifold.\\
    In the case when $\dim_\mathbb R \mathcal{M}_{d, k,d}(y)=0$, then it is a compact manifold.\\
    In the case when $\dim_\mathbb R\mathcal{M}_{d, k,d}(y)=1$, $\mathcal{M}_{d, k,d}(y)$ has a compactification by smooth manifolds of the form $\mathcal{M}_{d_1,k_1,m_1}(y_1)\times\mathcal{M}_{d_2,k_2,m_2}(y_2)$ such that
    \begin{enumerate}
    \item $d_1, d_2\geq 0$ are integers such that $d_1+d_2=d$,
    \item $k_1, k_2\geq0$ are integers such that $k_1+k_2=k+1$,
    \item $m_1, m_2\geq 0$ are integer such that $m_1+m_2=m$ and $m_i\equiv d_i$,
        \item $\dim_\mathbb R \mathcal{M}_{d_i,k_i,m_i}(y_i)=0$ for $i=1,2$.
    \end{enumerate}
\end{thm}
Now following our grading convention which is inline with \cite{perutz2022constructing}, \cite{sheridan2020versality}, the signed counts of rigid curves $u$ provides us with a linear isomorphism
\begin{equation*}
    \mu_u:|o_{y_1}|_\mathbb G\otimes\dots\otimes|o_{y_k}|_\mathbb G\rightarrow|o_{y_0}|_\mathbb G[2-k]
\end{equation*}
defining a $\Lambda$-linear map
\begin{equation*}
    m_k^{\text{rel}}:\hom^{\F(M,D)}(L_0,L_1)\hat{\otimes}_\Lambda\dots\hat{\otimes}_\Lambda\hom^{\F(M,D)}(L_{k-1}, L_k)\rightarrow\hom^{\F(M,D)}(L_0, L_k)
\end{equation*}
given by $m_k^{\text{rel}}:=\sum\mu_uq^{[u].[D]}$ where the sum is taken over all rigid curves $u$, which is convergent in the $q$-adic topology.

\begin{thm}[\cite{perutz2022constructing}, \cite{sheridan2020versality}]
    The operations $\{m_k^{\text{rel}}\}_{k\geq0}$ defines an $A_\infty$-structure on $\F(M,D)$ which is independent, up to filtered curved quasi-equivalence of $A_\infty$-categories, of the choices of Hamiltonian perturbations and stip-like ends. 
\end{thm}
 \begin{prop}
             The relative Fukaya category $\F(M, D)$ is a deformation of the compact exact Fukaya category $\F(X)$.
        \end{prop}

        \begin{proof}
            It suffice to show that there exists a strict isomorphism $\F(M,D)\hat{\otimes}_\Lambda\Lambda/\langle q_1, \dots, q_N\rangle\cong \F(X)$. Indeed, by definition we have the objects of $\F(M, D)$ are the same as that $\F(X)$. Given two objects $L_0, L_1$, we choose Floer datum $(H^{\text{exact}}_{0,1}, J^{\text{exact}}_{0,1})$ and $(H^{\text{rel}}_{0,1}, J^{\text{rel}}_{0,1})$ so that $H^{\text{rel}}_{0,1}\equiv H_{0,1}^{\text{exact}}$on a neighborhood of $D$ containing $\partial\overline X$. Using such Floer datum, morphism spaces in $\F(M, D)$ are precisely $\hom^{\F(M,D)}(L_0, L_1):=\hom^{\F(X)}(L_0, L_1)\hat{\otimes}_\mathbb C\Lambda$. As for the $A_\infty$-structure maps, We denote by $m_k^{\text{rel}}$ and $m_k^{\text{exact}}$ the corresponding $A_\infty$-operators on $\F(M, D)$ and $\F(X)$ respectively and observe that the $\textit{coefficient moduli space}$ of $m^{rel}_k\text{ (mod }q)$ is of the form of pseudo-holomorphic polygons $u$ satisfying $[u]=0\in H^2(M,X)$. By Lemma \ref{PositivityIntersection}, it follows that all such curves have images in $X$. Observe that, the space of compactly-supported Hamiltonians in $X$ is contained in $C^{\infty}(M,D)$ and that $\mathcal{J}\cap\mathcal{J}^{rel}\neq\emptyset$. As intersection of Baire sets is a Baire set, we can find generically auxiliary data for the construction of $\F(M,D)$ and auxiliary data for the construction $\F(X)$ such that $m_k^{\text{rel}}\text{ (mod }q)=m_k^{\text{exact}}$ for all integers $k\geq0 $ and for any collection of objects in any of $\F(M,D)$ and $\F(X)$.\\
            Now the result follows from the fact that the construction of $\F(M,D)$ and $\F(X)$ is independent of auxiliary choices up to quasi-equivalence. 
        \end{proof}

        \paragraph{Wrapped Fukaya Category of $X$:}
        We start by recalling the construction of wrapped Floer homology following \cite{abouzaid2010open}.
        \begin{defn}[Cylindrical Lagrangians]
            A Lagrangian $L\subset X$ is said to be cylindrical if $\theta|_L$ is an exact $1$-form on $L$ of the form $df=\theta|_L$ such that $\operatorname{supp}(f)\subseteq L\cap\{\rho\leq 1-\epsilon\}$. 
        \end{defn}
        \begin{example}
            Any closed exact Lagrangian in $\overline{X}$ is a cylindrical Lagrangian in $X$.
        \end{example}
        Let $L_0,L_1$ be two cylindrical Lagrangians each equipped with a brane structure. Fix a $[0,1]_t$-dependent set of admissible basic cofinal family of Hamiltonins $\{H'_n\}_{n\geq1}$ such that for every $n\geq 1$ the image of the $L_0$ under the time-$1$ flow of $X_{H'_n}$ intersects $L_1$ transversely. We denote by $CF^*(L_0,L_1;H'_n):=\bigoplus_{y\in\mathcal{X}(L_0, L_1; H'_n)}|o_{y}|_\mathbb G$. Let $(J_{n,t})_{t\in[0,1]}$ be time-dependent almost complex structure on $X$ such that $J_{n,t}$ is of contact-type as in Definition \ref{Contact-typeAlmostComplex} on $\{\rho\geq 1-\epsilon\}$. Now for two distinct Hamiltonian chords $y_-, y_+\in\mathcal X(L_0, L_1; H'_n)$ we consider smooth maps $u:\mathbb R_s\times[0,1]_t\rightarrow X$ such that:
        \begin{enumerate}
            \item $\partial_su+J_{n,t}(\partial_tu-X_{H'_n}(u))=0$,
            \item $u(., 0)\in L_0$, $u(., 1)\in L_1$,
            \item $\int_{\mathbb R\times [0,1]}|\partial_su|^2 ds\wedge dt<\infty$,
            \item $\lim_{s\rightarrow\pm\infty}u(s,t)=y_\pm(t)$. 
        \end{enumerate}
    Denote by $\tilde{\mathcal{M}}(y_-, y_+)$ the moduli space of such map and note that, for a generic choice of $(J_{n,t})_{t\in[0,1]}, \tilde{\mathcal{M}}(y_-, y_+)$ has the structure of an oriented finite-dimensional smooth manifold, where the $\mathbb R$-action given by $s_0.u(s,t)\mapsto u(s+s_0,t)$ is free. Denote by $\mathcal{M}(y_-, y_+):=\tilde{\mathcal{M}}(y_-,y_+)/\mathbb R$. We define the Floer differential by
    \begin{equation*}
        \partial^{WF}y_+:=\sum_{y_-\neq y_+\in\mathcal{X}(L_0, L_1; H'_n)} |\mathcal{M}(y_-, y_+)|_\mathbb G.y_-.
    \end{equation*}
Similarly as in the $\textit{closed case}$, we use domains with sprinkles to break the $\mathbb R$-symmetry in the definition of the continuation maps. Namely, let $H'_s$ be a generic smooth monotone homotopy from $H'_{n+1}$ to $H_n'$ and $J_s$ be a smooth homotopy from $J_{n+1}$ to $J_n$. Namely, $(H'_s, J_s)\equiv (H'_n, J_n)$ for $s>>1$ and $(H'_s, J_s)\equiv (H'_{n+1}, J_{n+1})$ for $s<<-1$. For Hamiltonian chords $y_n\in\mathcal{X}(L_0, L_1; H_n')$ $y_{n+1}\in\mathcal{X}(L_0, L_1; H_{n+1}')$, rigid counts of the moduli space $\mathcal{M}^{al}(y_{n}, y_{n+1})$ parametrizing smooth maps $u:\mathbb R_s\times[0,1]_t\rightarrow X$ satisfying:
\begin{enumerate}
    \item $\partial_su+J_s(\partial_tu-X_{H_s}(u))=0$,
    \item $u(., 0)\in L_0$, $u(., 1)\in L_1$,
    \item $\int_{\mathbb R\times [0,1]}|\partial_su|^2 ds\wedge dt<\infty$,
    \item $\lim_{s\rightarrow-\infty}u(s,t)=y_{n+1}(t)$ and $\lim_{s\rightarrow\infty}u(s,t)=y_{n}(t)$,
\end{enumerate}
defines a linear map $c(y_{n}):=\sum_{y_{n+1}\in\mathcal{X}(L_0, L_1; H'_{n+1})} |\mathcal{M}^{al}(y_{n}, y_{n+1})|_\mathbb G.y_{n+1}$, defining a continuation chain map $c:CF^*(L_0,L_1;H'_n)\rightarrow CF^*(L_0,L_1;H'_{n+1})$. The cochain complex of the wrapped Floer homology is the telescopic complex given by
\begin{equation*}
    \mathcal{W}(L_0, L_1;\{H'_n\}_{n\geq1}):=\bigoplus_{n\geq 1}CF^*(L_0,L_1;H'_n)[\tau],
\end{equation*}
where $\tau$ is a formal variable of $\deg(\tau)=-1$ and $\tau^2=0$. Its differential is given by
\begin{equation*}
    m_1^{\text{wrap}}(y+\tau y'):=(-1)^{|y|}\partial^{WF}y+(-1)^{|y'|}(\tau\partial^{WF}y'+c(y')-y'). 
\end{equation*}
The homology $H^*(\mathcal{W}(L_0, L_1;\{H'_n\}_{n\geq1}), m_1^{\text{wrap}})\equiv HW^*(L_0, L_1)$ is independent, up to linear isomorphism, of the choices $\{H'_n\}_{n\geq1}$ or $\{J_n\}_{n\geq1}$. With this in hand, the construction of the wrapped Fukaya category follows the same steps as above.\\

Objects of $\mathcal W(X)$, the wrapped Fukaya category of $X$, are cylinderical Lagrangians equipped with brane structure. Given two objects $L_0, L_1$ of $\mathcal W(X)$ we first choose a generic Floer datum $\{H'_n\}_{n\geq1}$ and $\{J_n\}_{n\geq1}$ as above and define their morphim space by $\hom^{\mathcal W(X)}(L_0,L_1):=\mathcal{W}(L_0, L_1;\{H'_n\}_{n\geq1})$. As for the $A_\infty$-structure, given a Lagrangian label $\mathbb L=(L_0, \dots, L_k)$ for $k\geq0$, we fix a consistent universal choice of strip-like ends $\epsilon_0, \epsilon_1, \dots, \epsilon_k$ and generic Floer datum as above, having no triple intersections. We fix a consistent universal choice of perturbation $(K_r, J_r)$ for $r\in\R^{\text{disk}, al}_{k,p, n}$. Let $y_i\in\mathcal{X}(L_i, L_{i+1})$ and write $r=[\Sigma; z_0, \dots,z_k;\psi]$, we set $C_r:=\Sigma\setminus\{z_0, \dots, z_k\}$ and consider smooth maps $u:C_r\rightarrow X$ such that
\begin{enumerate}
    \item $(du-Y)^{0,1}=0$,
    \item for every $z\in\partial_iC_r$, $u(z)\in L_i$,
    \item $\lim_{s\rightarrow-\infty}u(\epsilon_0(s,t))=y_0(t)$ and $\lim_{s\rightarrow\infty} u(\epsilon_i(s,t))=y(t)$. 
\end{enumerate}
Following Sections $2.5, 2.6$ of \cite{abouzaid2010open}, for a generic choice of perturbation data, the moduli space $\mathcal{M}^{\text{disk}, al}_{k, p, n}(y):=\{(r, u)\}$, where $r\in\R^{\text{disk}, al}_{k,p, n}$ and $u$ is as above, has the structure of an oriented finite-dimensional smooth manifold. Counting rigid curves provides us with a linear map
\begin{equation*}
    \tilde{m}^{\text{wrap}}_k:\hom^{\mathcal{W}(X)}(L_0, L_1)\otimes\dots\otimes\hom^{\mathcal W(X)}(L_{k-1}, L_k)\rightarrow \bigoplus_{n\geq1}CF^*(L_0,L_1;\{H'_n\}_{n\geq1}),
\end{equation*}
with a unique $\tau$-linear extension, commuting with $\partial_\tau$. Similarly as above, we set
\[
m^{\text{wrap}}_k(y_1, \dots, y_k):=
\begin{cases}
    \tilde{m}_k^{\text{wrap}}(y_1, \dots, y_k) & \text{for } k\geq2,\\
    \tilde{m}_1^{\text{wrap}}(y_1)+(-1)^{|y_1|}\partial_\tau y_1 & \text{for } k=1. 
\end{cases}
\]
\begin{thm}[\cite{abouzaid2010open}]
    The operations $\{m^{\text{wrap}}_k\}_{k\geq 0}$ gives $\mathcal{W}(X)$ the structure of an $A_\infty$-category. Moreover such structure is independent, up to quasi-equivalence of $A_\infty$-categories, of the choice of stip-like ends, Floer datum, and perturbation datum.
\end{thm}

\subsection{Closed-Open Map} \label{COconstruction}     
Consider $\mathcal W(X)$ the wrapped Fukaya category of $X$. We are interested in studying its $\textit{deformations}$ as an $A_\infty$-category and for which we recall the following notions.

\begin{defn}[Hochschild Cochain Complex]
    We define the Hochschild cochain complex of $\mathcal W(X)$ by
    \begin{equation*}
       \hspace{-1cm} CC^*(\mathcal{W}(X)):=\Pi_{L_0,\dots, L_k}\hom_\mathbb C(\hom^{\mathcal{W}(X)}(L_0, L_1)\otimes\dots\otimes\hom^{\mathcal W(X)}(L_{k-1}, L_k),\hom^{\mathcal{W}(X)}(L_0, L_k))[-k]. 
    \end{equation*}
\end{defn}
  $CC^*(\mathcal W(X))$ comes with a $\textit{Gerstenhaber}$ product structure given by
  \begin{equation*}
      (\phi_1\circ\phi_2)^l(a_l,\dots,a_1):=\sum_{1\leq i,j\leq l}(-1)^*\phi_1^{l-j+1}(a_l,\dots,a_{i+j-1},\phi_2(a_{i+j},\dots,a_{i+1}),\dots,a_1),
  \end{equation*}
where $*=(\deg(\phi_2)-1)(\sum_{1\leq j\leq i}\deg(a_j)-1)$. Moreover, $CC^*(\mathcal{W}(X))$ also comes equipped with a $\textit{bracket}$ operation, given by
\begin{equation*}
    [\phi_1,\phi_2]:=\phi_1\circ\phi_2-(-1)^*\phi_2\circ\phi_1,
\end{equation*}
    where $*=(\deg(\phi_1)-1)(\deg(\phi_2)-1)$. In such context, $A_\infty$-structures on $\mathcal{W}(X)$ are equivalent to a choice of Hochschild cohomology class $m\in CC^2(\mathcal W(X))$ such that $m\circ m=0$. Such $m$ gives rise to a differential 
    \begin{equation*}
        \partial\phi:=[m, \phi],
    \end{equation*}
making $(CC^*(\mathcal W(X))[1], \circ, [.,.], \partial)$ a differential graded Lie algebra.\\

Let $\mathbb L=(L_0, \dots, L_k)$ be a Lagrangian label. Let $\{H_n\}_{n\geq 1}$ and $\{H'_n\}_{n\geq 1}$ be two admissible cofinal families of basic Hamiltonians as above and fix a consistent universal choice of ends and perturbations $(K, J)$ on $\R^{\text{disk}, al}_{d, k}$. Given a $p$-flavour and a tuple of $H_n$-Hamiltonian orbits $x$ and a tuple of $H'_n$-Hamiltonian chords $y$. For $r\in\R^{\text{disk}, al}_{d,k, p, n}(\mathbb L)$ written as $r=[\Sigma; z_0, \dots,z_k;\zeta_1, \dots,\zeta_d;\psi]$ we set $C_r:=\Sigma\setminus\{z_0, \dots,z_k; \zeta_1, \dots, \zeta_d\}$. Consider smooth maps $u:C_r\rightarrow X$ such that
\begin{enumerate}
    \item $(du-Y)^{0,1}=0$,
    \item $\lim_{s\rightarrow-\infty}u(\epsilon_{z_0}(s,t))=y_0(t)$ and $\lim_{s\rightarrow-\infty}u(\epsilon_{\zeta_0}(s,t))=x_0(t)$,
    \item $\lim_{s\rightarrow\infty}u(\epsilon_{z_i}(s,t))=y_i(t)$ and  $\lim_{s\rightarrow\infty}u(\epsilon_{\zeta_i}(s,t))=x_i(t)$ for $i\neq 0$.
\end{enumerate}
We denote by $\mathcal{M}_{d, k,p,n}(x,y):=\{(r,u)\}$ the moduli space of all such tuples such that $r\in\R^{\text{disk}, al}_{d, k, p, n}(\mathbb L)$ and $u$ is a smooth map as above. 
\begin{thm}[\cite{abouzaid2010open}, \cite{borman2024l_}, \cite{strom2024maurer}]
    There exists a generic consistent, equivariant universal choice of perturbation data such that, $\mathcal{M}_{d,k,p,n}(x,y)$ has the structure of an oriented finite-dimensional smooth manifold.\\
    In the case of $\dim_\mathbb R\mathcal{M}_{d, k,p,n}(x,y)=0$, then it is a compact oriented manifold. 
\end{thm}

The signed count of rigid curves gives us a linear map
\begin{equation*}
    co_{d, k,p,n}:t^{i_1}|o_{x_1}|_\mathbb G\otimes\dots\otimes t^{i_d}|o_{x_d}|_\mathbb G\otimes t^{j_1}|o_{y_1}|_\mathbb G\otimes\dots\otimes t^{j_k}|o_{y_k}|_\mathbb G\rightarrow |o_{y_0}|_\mathbb G][2-2d-k],
\end{equation*}
where $i_s=p^{-1}(s)$ and $j_s=p^{-1}(d+s)$. Summing over all possibilities of $(p, n)$, we get a linear map
\begin{equation*}
    co_{d,k}:SC^*(X)^{\otimes d}\otimes\hom^{\mathcal{W}(X)}(L_0, L_1)\otimes\dots\otimes\hom^{\mathcal{W}(X)}(L_{k-1}, L_k)\rightarrow\bigoplus_{n\geq 1} CF^*(L_0, L_l; H'_n)[2-2d-k]. 
\end{equation*}
Following Section $6$ of \cite{borman2024l_}, $co_{d, k}$ has a unique $\tau$-linear extension to $\hom^{\mathcal W(X)}(L_0, L_l)$ commuting with $\partial_\tau$. After summing over all $(d, k)$ and using the fact that by exactness of $X$, such sum is actually a finite sum, we get the following:
\begin{thm}[\cite{borman2024l_}]\label{L_inftyMorphism}
    $\operatorname{CO}:(SC^*(X;\{H_n\}_{n\geq1})[1], \ell_*)\rightarrow (CC^*(\mathcal W(X))[1], \circ, [.,.], \partial)$ defines an $L_\infty$-morphism, such that $\operatorname{CO}_0=\tilde{m}^{\text{wrap}}$. 
\end{thm}

In this paper, we are interested in the full sub $A_\infty$-category of $\mathcal W(X)$, denoted by $\mathcal{W}(M,D)$, formed from closed exact Lagrangian branes in $\overline{X}$ where its $A_\infty$-structure is given by $m^{\text{wrap}, CO(\beta)}$, the $CO(\beta)$-deformed operator. Namely, the objects of $\mathcal{W}(M,D)$ are by definition, the same objects of $\F(X)$. Given two objects $L_0, L_1$ of $\mathcal W(M,D)$ we a pick Floer datum $(H_{0,1}, J_{0,1})$ as in the construction of $\hom^{\F(X)}(L_0, L_1)$ such that $J_0$ the reference almost complex structure is in $\mathcal{J}(M,D)$, and we extend smoothly $H_{0,1}$ to be in $C^\infty(M,D)$ on $\{\rho\geq1+\epsilon\}$. For such Floer datum we get $\hom^{\mathcal W(M,D)}(L_0,L_1)\equiv\hom^{\F(X)}(L_0,L_1)\hat{\otimes}_\mathbb C\Lambda$. Moreover, the $A_\infty$-structure on $\mathcal W(M,D)$ is given by
\[
m^{\text{wrap}}+\sum_{d\geq1}\frac{1}{d!}\operatorname{CO}_d(\beta,\dots,\beta).
\]

\section{Proof of the main theorem}

Denote by $\mathcal W_{cpt}(X)\cong\mathcal F(X)$ the full sub-$A_\infty$ category generated by compact objects in $\overline{X}$ and recall that 
\[
\mathfrak{g}_{>a}\equiv \mathcal P_{>a}SC^*(X)\oplus\bigoplus_{j=1}^N\mathbb C[\tau].|o_{x^{D_j}}|_\mathbb G
\]
with its given partial $L_\infty$-structure as in Definition \ref{PartialLinfty}.

\begin{lemma}[Key Lemma]\label{keylemma}
    There exists a graded-symmetric, $\mathcal P$-filteration preserving maps
    \[
    \operatorname{CO}^D_d:\g^{\otimes d}_{>0}\rightarrow CC^*(\mathcal W_{cpt}(X))[2-2d]
    \]
    for $d\geq1$, such that:
    \begin{enumerate}
        \item The $L_\infty$-morphism identities holds on $\g_{>1}$,
        \item $\operatorname{CO}^D_d|_{\mathcal P_{>0}SC^*(X)^{\otimes d}}\equiv \operatorname{CO}_d$ as in Theorem \ref{L_inftyMorphism}.
    \end{enumerate}
\end{lemma}
\begin{proof}
Let $\mathbb L=(L_0, \dots, L_k)$ be a Lagrangian label of compact exact Lagrangian branes. Let $\{H_n\}_{n\geq 1}$ and $\{H'_n\}_{n\geq 1}$ be two admissible cofinal families of basic Hamiltonians as above and fix a consistent universal choice of ends and perturbations $(K, J)$ on $\R^{\text{disk}, al}_{d, k}$. Given a $p$-flavor, $p:F\rightarrow\{1, \dots, d+k\}$, $n=(n_1, \dots, n_d)$ a tuple of weights, a tuple $x=(x_1, \dots, x_d)$ such that $x_i\in\g_{>0}$ and a tuple of $H'_n$-Hamiltonian chords $y=(y_0, y_1, \dots, y_k)$. For $r\in\R^{\text{disk}, al}_{d,k, p, n}(\mathbb L)$ written as $r=[\Sigma; z_0, \dots,z_k;\zeta_1, \dots,\zeta_d;\psi]$ we set $C_r\setminus\{z_0, \dots,z_k; \zeta_i\}$ such that $i\in Z^c(n)$. Consider smooth maps $u:C_r\rightarrow M$ such that
\begin{enumerate}
    \item $(du-Y)^{0,1}=0$,
    \item $\lim_{s\rightarrow-\infty}u(\epsilon_{z_0}(s,t))=y_0(t)$ and $\lim_{s\rightarrow-\infty}u(\epsilon_{\zeta_0}(s,t))=x_0(t)$,
    \item $\lim_{s\rightarrow\infty}u(\epsilon_{z_i}(s,t))=y_i(t)$ and for each $i\in Z^c(n)$, $\lim_{s\rightarrow\infty}u(\epsilon_{\zeta_i}(s,t))=x_i(t)$,
    \item $[u].D=|Z(n)|$ and $u^{-1}(D)=\{\zeta_i:i\in Z(n)\}$. 
\end{enumerate}
We denote by $\mathcal{M}_{d, k,p,n}(x,y):=\{(r,u)\}$ the moduli space of all such tuples such that $r\in\R^{\text{disk}, al}_{d, k, p, n}(\mathbb L)$ and $u$ is a smooth map as above. By Theorem \ref{L_inftyPerturbationData}, Lemma $5.6, 5.9$ and Corollary $5.10$ of \cite{perutz2022constructing} and using the fact that intersection of Baire sets is a Baire set, it follows that generically, $\mathcal{M}_{d, k,p,n}(x,y)$ is a smooth oriented manifold when $\dim_\mathbb R\mathcal{M}_{d, k,p,n}(x,y)\in\{0,1\}$.\\

$\textbf{Step 1:}$ No $D$-orbits.

    The key idea is to consider each connected component of a given $2$-colored tree after removing its disk sub-tree and use the energy estimates as found in \cite{strom2024maurer} due to our choice of Hamiltonians together with the index sum formula. Indeed, let $\{u_i\}_{i\geq1}\subseteq\mathcal{M}_{d,k,p,n}(x,y)$ and suppose that $u_\infty$ is its stable limit indexed by a $2$-colored tree $T$. Following the same terminology as in \cite{strom2024maurer}, we \textit{prune} every internal dashed edge of weight-zero, as in Figure $2$ and Lemma $3.6$ of \cite{strom2024maurer} and denote the resulting tree by $T'$. For $e$ a dashed leaf of $T'$ of weight-zero, we denote by $B_e$ the bubble tree removed. We give each weight-zero leaf the auxiliary cap convention, as in Equation $(3.25)$ of \cite{strom2024maurer}. Namely,
    \[
    |\hat{x}_e|=2, \mathcal A(\hat{x}_e)=0, \text{and }\mathcal P(\hat{x}_e)=-2,
    \] 
    where $\hat{x}_e$ denotes its capped Hamiltonian orbit. For $v$ a vertex of $T'$ we set, $\rho_v:=2d_v+k_v+|F_v|-2$ if $v$ is a solid vertex and $\rho_v:=2d_v+|F_v|-3$ if $v$ is dashed vertex, where $d_v, k_v$ are the numbers of outgoing dashed and solid edges respectively and $|F_v|$ is the number sprinkles carried by the component corresponding to $v$. We decompose the vertices of $T'$ as follows:
    \begin{enumerate}
        \item $V^-$ the set of all dashed vertices $v$ of $T'$ such that $\mathcal{P}(\hat{x}_{in(v)})<0$ and no other dashed vertex on the minimal path joining $v$ to the solid sub-tree passes through such vertex.
        \item For $v\in V^-$, let $U_v$ be the entire connected dashed sub-tree rooted at $v$ and set $U:=\bigcup_{v\in V^-}U_v$.
        \item Let $I:= V(T')\setminus U$.
    \end{enumerate}

    $\textit{Case 1:}$ Suppose that $v\in V^s(T')$ is a solid vertex.\\
    Denote by $y_{in(v)}$ the (unique) incoming Hamiltonian chord corresponding to $u_v$ and let $y_e$ be a Hamiltonian chord of an outgoing solid edge adjacent to $v$. By transversality of our moduli spaces, we have $\operatorname{ind}(u_v)\geq 0$. Using Equations $(3.20)$ and $(3.26)$ of \cite{strom2024maurer}, it follows that,
    \begin{equation}
        \rho_v+|y_{in(v)}|-\sum_{e\in E^s_{out}(v)}|y_e|\geq \sum_{e\in E^d_{out}(v):n_e>0}|\hat{x}_e|+2\sum_{e\in E^d_{out}(v):n_e=0}([u_{v_e}].D+c_1(B_e)-1).
    \end{equation}
     Notice that, by positivity of intersections Lemma \ref{PositivityIntersection} and the monotonicity assumption on $M$, it follows that $u_v.D+c_1(B_e)\geq 1$ and our auxiliary capping convention, it follows that
     \begin{equation}
         \rho_v+|y_{in(v)}|-\sum_{e\in E^s_{out}(v)}|y_e|\geq\sum_{e\in E^d_{out}(v)}|\hat{x}_e|.
     \end{equation}
     
$\textit{Case 2: }$Suppose that $v\in I\cap V^d(T')$.\\ 
Then all edges corresponds to $\mathcal P$-non negative Hamiltonian orbits and hence,
\begin{equation}
    \rho_v+|\hat{x}_{in(v)}|\geq \sum_{e\in E_{out}(v)}|\hat{x}_e|.
\end{equation}

$\textit{Case 3:}$ Suppose that $v$ is a vertex of $U$.\\
We first note that, by additivity of weights and actions, we have the following:
\begin{equation}
    n_{in(v)}=\sum_{e\in E_{out}(v)}n_e+|F_v|,
\end{equation}
\begin{equation}
    \mathcal A(\hat{x}_{in(v)})\geq\mathcal \sum_{e\in E_{out}(v)}\mathcal A(\hat{x}_{e})+C_{n_v}|F_v|.
\end{equation}
If $v\in V^-$, it follows by the definition of $\mathcal P$ (c.f \ref{P-filtrationResult}) that $|\hat{x}_{in(v)}|>E(\hat{x}_{in(v)})$. Thus, by the estimates $(3.29), (3.30)$ of \cite{strom2024maurer} (c.f. Item 1 of Theorem \ref{L_inftyPerturbationData}), it follows that
    \begin{equation}
        -(C+\frac{1}{\kappa}C_{n_v})|F_v|+|\hat{x}_{in(v)}|\geq \sum_{e\in E_{out}(v)}(Cn_e+\frac{1}{\kappa}\mathcal A(\hat x_e))\equiv\sum_{e\in E_{out}(v)}E(\hat{x}_e).
  \end{equation}
    The same is true at every descendant vertex $v$ (i.e. $v\in U\setminus V^-$), and hence we have
    \[
    -(C+\frac{1}{\kappa}C_{n_v})|F_v|+E(\hat x_{in(v)})\geq \sum_{e\in E_{out}(v)}E(\hat x_e)).
    \]
Now we can combine the above estimates to get our claim. Indeed, the key point of the auxiliary cap convention is that we have the following dimension equality
    \[
    \dim_\mathbb R\R^{\text{disk}}_{d',k,p'}+|y_0|-\sum_{1\leq i\leq k}|y_i|-\sum_{e\in E^d(T')}|\hat{x}_e|\equiv\dim_\mathbb R\R^{\text{disk}}_{d,k,p}+|y_0|-\sum_{1\leq i\leq k}|y_i|-\sum_{1\leq i\leq d :n_i>0}|x_i|-2|Z(n)|,
    \]
    where both sides are equal to $\dim_\mathbb R\mathcal M_{d, k,p, n}(x,y)$ and $d', p'$ denote the resulting interior leafs and flavors after pruning and that the Fredholm index of the solid sub-tree satisfies
    \[
    \sum_{v\in V^s(T')}(|y_{in(v)}|-\sum_{e\in E^s_{out}(v)}|y_e|)=|y_0|-\sum_{1\leq i\leq k}|y_i|.
    \]
    Now noting that,
    \[
    \sum_{v\in V(T)}\rho_v=2d+k+|F|-2-|iE(T)|,
    \]
    and in particular,
    \[
    \sum_{v\in V(T)}\rho_v=\dim_{\mathbb R}\mathcal R^{\text{disk}}_{d, k,p}-|iE(T)|,
    \]
 (c.f. Equations $(2.12), (2.21), (2.22)$ of \cite{borman2024l_}), together with the index sum formula, it follows that 
 \begin{equation}
     \sum_{v\in U}(2d_v-2+(1+C+\frac{1}{\kappa}C_{n_v})|F_v|+\sum_{e}\mathcal P(\hat{x}_e))\leq \dim_\mathbb R\mathcal M_{d,k,p,n}(x,y)+1-|I|,
 \end{equation}
where the inner sum is taken over all edges $e$ adjacent to $v$ meeting a $\mathcal P$-negative orbit. 

Now by Equation $(3.14)$ of \cite{strom2024maurer}, we have $1+C+\frac{1}{\kappa}C_{n_v}>3$ and as $\dim_{\mathbb R}\mathcal M_{d,k,p,n}(x,y)\in\{0,1\}$ and $|V^s(T')|\geq1$, it follows using the same argument as in the proof of Lemma $3.6$ of \cite{strom2024maurer} that, if $\dim_\mathbb R\mathcal M_{d, k,p,n}(x,y)=0$ and all $x_i\in\g_{>0}$ then, $u_\infty$ has no components asymptotic to a $\mathcal P$-negative Hamiltonian orbit. The same is true in the case $\dim_\mathbb R\mathcal M_{d,k,p,n}(x,y)=1$ and all $x_i\in\g_{>1}$.\\
    
    $\textbf{Step 2:}$ No sphere bubbles of Chern number $1$.

Using the same notation as in $\textbf{Step 1}$ and the index sum formula, it follows that, if $\dim_\mathbb R \mathcal{M}_{d, k,p,n}(x, y)\in\{0,1\}$ then,
\[
2\sum_{e:n_e=0}([u_v].D+c_1(B_e)-1)\leq1-|iE(T')|.
\]
Note that, by monotonicity of $M$ it follows that $c_1(B_e)\geq 1$. On the other hand, by positivity of intersection Lemma \ref{PositivityIntersection}, it follows that, the only case we need to consider is when $[u_v].D=0$ and $c_1(B_e)=1$. In particular, the sphere bubble in $D$ is attached to a constant Hamiltonian orbit or to a disk tangent to $D$. Hence, by the proof of Lemma $3.4$ and $3.8$ of \cite{strom2024maurer} (c.f. Theorem \ref{CompactnessResult}) or Corollary $5.13$ and Lemma $5.16$ of \cite{perutz2022constructing}, respectively, it follows that if $u_\infty$ is the stable limit of a sequence $u_i\in\mathcal{M}_{d,k,p,n}(x,y)$ such that $\dim_\mathbb R\mathcal{M}_{d, k,p, n}(x,y)\in\{0,1\}$ then, $u_\infty$ has no spherical components.\\

    From Steps $1,2$, it follows that, if $\dim_\mathbb R\mathcal{M}_{d, k,p, n}(x, y)=0$ then, $\mathcal M_{d,k,p,n}(x,y)$ is compact. While, if $\dim_\mathbb R\mathcal{M}_{d,k,p,n}(x,y)=1$ then, its Gromov's compactification only has 2-colored codimension-$1$ breaking. Namely, breaking corresponding to a Hamiltonian orbit in $X$ or a disk breaking.\\
    
    $\textbf{Step 3: }\operatorname{CO}^D_d$ is well-defined.\\
    Let,
    \[
    \operatorname{co}^D_{d,k}:\g^{\otimes d}_{>0}\otimes\hom^{\mathcal W_{cpt}(X)}(L_0, L_1)\otimes\dots\otimes\hom^{\mathcal W_{cpt}(X)}(L_{k-1}, L_k)\rightarrow\hom^{\mathcal W_{cpt}(X)}(L_0, L_k)[2-2d-k],
    \]
    be the signed count of rigid elements of $\mathcal M_{d, k,p,n}(x,y)$ summed over all $(p, n)$. We give $\operatorname{co}^D_{d,k}$ the unique $\tau$-extension as given by Equations $(5.28), (5.29)$ of \cite{borman2024l_}. By the above boundary stratification of $\mathcal{M}_{d,k,p,n}(x,y)$ of expected dimension $0$ or $1$, it follows that $\operatorname{co}^D_{d,k}$ is well-defined.\\
    
    $\textbf{Step 4: }$ Proof of Item $1$.

    Let $x_1, \dots, x_d\in\g_{>1}$. Using the convention $(5.23)$ of \cite{borman2024l_}, Definition \ref{PartialLinfty} (c.f. Equations $(4.15)-(4.18)$ of \cite{strom2024maurer}) and the boundary stratification of $\mathcal{M}_{d,k,p,n}(x,y)$ when $\dim_{\mathbb R}\mathcal M_{d,k,p,n}(x,y)=1$ as described above, it follows that
  \begin{equation*}
\begin{aligned}
0={}&
\sum_{j=1}^{d}
\sum_{\sigma\in\operatorname{Unsh}(j,d)}
(-1)^{\epsilon(\sigma)}
\operatorname{CO}^{D}_{d-j+1}
\Bigl(
    \ell_j(x_{\sigma(1)},\ldots,x_{\sigma(j)}),
    x_{\sigma(j+1)},\ldots,x_{\sigma(d)}
\Bigr)
\\
&+
\Bigl[
    m^{\mathrm{wrap}},
    \operatorname{CO}^{D}_{d}(x_1,\ldots,x_d)
\Bigr]
\\
&+
\sum_{\substack{
    1\leq j\leq d-1\\
    \sigma\in\operatorname{Unsh}(j,d)\\
    \sigma(1)<\sigma(j+1)
}}
(-1)^{
    \epsilon(\sigma)
    +\sum_{i=1}^{j}|x_{\sigma(i)}|
}
\\
&\qquad\cdot
\Bigl[
    \operatorname{CO}^{D}_{j}
    (x_{\sigma(1)},\ldots,x_{\sigma(j)}),
    \operatorname{CO}^{D}_{d-j}
    (x_{\sigma(j+1)},\ldots,x_{\sigma(d)})
\Bigr].
\end{aligned}
\end{equation*}
    where $\epsilon(\sigma)$ is the Koszul sign and 
 \[
\operatorname{Unsh}(j,d)
:=
\left\{
\sigma\in \operatorname{Sym}(d)
\;\middle|\;
\begin{aligned}
\sigma(1)&<\cdots<\sigma(j),\\
\sigma(j+1)&<\cdots<\sigma(d)
\end{aligned}
\right\}.
\]   
    Indeed, the terms \[
    \operatorname{CO}^D_{d-j+1}(\ell_j(x_{\sigma(1)},\dots, x_{\sigma(j)}),x_{\sigma(j+1)},\dots, x_{\sigma(d)})\] correspond to when we have a $2$-colored tree with a single dashed edge (i.e. breaking at a Hamiltonian orbit), while the terms \[
    [\operatorname{CO}^D_{j}(x_{\sigma(1)},\dots, x_{\sigma(j)}),\operatorname{CO}^D_{d-j}(x_{\sigma(j+1)}, \dots, x_{\sigma(d)})], \Bigl[
    m^{\mathrm{wrap}},
    \operatorname{CO}^{D}_{d}(x_1,\ldots,x_d)
\Bigr] \]
    corresponds to a $2$-colored tree with a single solid internal edge (i.e disk breaking). As for the sign check, we use the convention as in Section $6$ of \cite{borman2024l_}. Namely, for a $\mathbb C$-vectorspace, we set
    \[
    \lambda(V):=\bigwedge^{\dim V}(V[-1])
    \]
    and use the Koszul sign convention as in $(6.1)-(6.4)$ of \cite{borman2024l_}. With this in mind, the only point need to be checked is the sign given to the rigid curves having tangency condition on $D$, Indeed, let $u$ be such curve and suppose that the $u$ has a tangency vector at an interior point $\zeta_i$ with $D$ of the form $(m_{i,j})_{j=1, \dots, N}$ and suppose that $m_{i,j}$ is either $0$ or $1$. Let, 
    \[
    V:=\bigoplus_{i\in Z(n)}(\nu_M(D_j))_{u(\zeta_i)}.
    \]
    \\
    We give $u$ an orientation-line so that $o_u\otimes\lambda(V)\cong \lambda(T_u\mathcal M(x',y))\otimes_\mathbb C(\bigoplus_i \lambda(T_{\zeta_i}\Sigma))$, where $x'$ is formed from $x$ after forgetting all forgetful marked points and $\Sigma$ is the domain of $u$. In this case, the formal generators $x^{D_j}$ are given trivial $\mathbb G$-normalized orientation lines and hence we can use the arguments in Section $6$ of \cite{borman2024l_}.\\
    
    $\textbf{Step 5: }$ Proof of Item $2$.

    Suppose that $x_1, \dots, x_d\in\mathcal P_{>0}SC^*(X)$, that is, $Z(n)=\emptyset$ and hence, $[u].D_j=0$ for every $j=1,\dots, N$. By the positivity of intersection Lemma \ref{PositivityIntersection} and the forgetfulness property of our auxiliary data as in Theorem \ref{L_inftyPerturbationData}, it follows that \[
    \operatorname{CO}^D_d(x_1, \dots, x_d)\equiv\operatorname{CO}_d(x_1, \dots, x_d).
    \] 
\end{proof}
\begin{thm}
			The full sub-category $\mathcal{W}(M, D)$, is filtered quasi-equivalent to $\mathcal{F}(M, D)$, Seidel's relative Fukaya category. 
		\end{thm}
        \begin{proof}
            Under the assumption that we can find generic $\textit{auxiliary data}$ for the construction of both of $\F(M,D)$ and $\mathcal W(M,D)$, in particular, by adhering to Theorem \ref{L_inftyPerturbationData}, Lemma $5.6, 5.9$ and Corollary $5.10$ of \cite{perutz2022constructing} and using the fact that intersection of Baire sets is a Baire set, we argue as follows:\\
            Denote by $\mathcal W^\alpha(M,D)$ the $\operatorname{CO}(\alpha)$-deformation of the full $A_\infty$-sub-algebra generated by closed, exact Lagrangians in $\overline{X}$.\\
            
            $\textbf{Step 1: }\textit{$\mathcal{W}^\alpha(M,D)$ is filtered quasi-equivalent to $\mathcal{W}(M,D)$}$.\\
            The key point is that $\alpha\sim\beta$ are gauge-equivalent and that $\operatorname{CO}^D$ is satisfies the $L_\infty$-morphism identities when restricted to $\g_{>1}$. Indeed, let $\mathcal C$ be the $A_\infty$-category whose objects are the same as $\F(X)$. For two objects $L_0, L_1$ of $\mathcal C$, we set $\hom^\mathcal C(L_0, L_1):=\hom^{\F(M,D)}(L_0, L_1)\hat{\otimes}\mathbb C[t, dt]$. The $A_\infty$-structure on $\mathcal C$ is given by an element in $m^\mathcal C\in CC^2(\mathcal W_{cpt}(X)\otimes\mathbb C[t,dt])^\Lambda=CC^2(\mathcal W_{cpt}(X))^\Lambda\hat{\otimes}\mathbb C[t,dt]$, where
            \begin{equation*}
                m^t:=m^{\text{wrap}}+\sum_{d\geq1}\frac{1}{d!}\operatorname{CO}^D(\beta(t), \dots, \beta(t)),
            \end{equation*}
            where $\beta(t)$ is as in Proposition \ref{GaugePathProp}, 
            after $t$ and $q$-linearily extending $\operatorname{CO}^D$ to a $\Lambda[t, dt]$-linear map, where the dg-Lie structure on $\mathbb C[t, dt]$ is given by $d(t)=dt$ and $d(dt)=0$. Note that, using Lemma $4.10$ and Corollary $4.11$ of \cite{strom2024maurer} (c.f. Proposition \ref{GaugePathProp}), we have $\beta(t)=\alpha(t)+\gamma(t)dt$ where $\alpha(t), \gamma(t)\in \g_{>1}^\Lambda$ for every $t\in[0,1]$. Using Item $1$ of Lemma \ref{keylemma} and the fact that $\beta(t)\in\operatorname{MC}(\g^\Lambda_t)$, it follows that, $m^t$ defines a class of degree 2 such that $d_{dR}m^t+\frac{1}{2}[m^t, m^t]=0$ and hence, $\mathcal C$ is a curved $A_\infty$-category. Now we claim that for $c\in\mathbb [0,1]$, $ev_c:\mathfrak g_t^\Lambda\rightarrow\mathfrak g^\Lambda$ given by $ev_c(t)=c$ and $ev_c(dt)=0$, induces a filtered quasi-equivalence between $\mathcal C\simeq\mathcal{W}^\alpha(M,D)$ and $\mathcal{C}\simeq\mathcal W(M,D)$. Indeed, as in the proof of Lemma $4.9$ in \cite{strom2024maurer}, $ev_c:\mathfrak g_t^\Lambda\rightarrow\mathfrak g^\Lambda$ is a filtraion-preserving $L_\infty$-morphism, and using Proposition \ref{GaugePathProp}, $\beta(t)\in\operatorname{MC}(\mathfrak g_t^\Lambda)$. Therefore, $ev_c$ induces a filtered $A_\infty$-functor from $\mathcal{C}$ to the $A_\infty$-category $\mathcal C_c$ of objects same as $\F(X)$, morphism spaces given by $\hom^{\mathcal C_c}(L_0, L_1)=\hom^{\F(M,D)}(L_0,L_1)$ and $A_\infty$-structure given by 
            \[
            m^{c}:=m^{\text{wrap}}+\sum_{d\geq1}\frac{1}{d!}\operatorname{CO}^D(\beta(c), \dots, \beta(c))\in CC^2(\mathcal W_{cpt}(X))^\Lambda.
            \]
            It suffice to show that, $ev_c$ induces a cohomologically fully-faithful functor as objects of $\mathcal C$ and $\mathcal C_c$ are the same. To do so and following the proof of Lemma $4.9$ of \cite{strom2024maurer}, we consider the associated graded $(\hom^{gr_*\mathcal C}(L_0, L_1), gr_*m^\mathcal C)$ and $(\hom^{gr_*\mathcal C_c}(L_0, L_1), gr_*m^{\mathcal C_c})$. As $\beta(t)\in\mathfrak g_{\geq1, t}^\Lambda$ and by Lemma \ref{PositivityIntersection} and Lemma \ref{MaxPrinciple}, it follows that on the first page of the associated spectral sequence, $ev_c$ induces a chain map of the form:
            \begin{equation*}
                (gr_*\hom^{\mathcal F(X)}(L_0,L_1), m_1^{\text{exact}})\otimes(\mathbb C[t, dt],d)\xrightarrow{\id\otimes ev_c}(gr_*\hom^{\mathcal F(X)}(L_0,L_1), m_1^{\text{exact}})\otimes(\mathbb C, 0).
            \end{equation*}
            As $ev_c:(\mathbb C[t, dt], d)\rightarrow(\mathbb C, 0)$ is a quasi-iosmorphism, it follows that $\id\otimes ev_c$ is a quasi-isomorphism. By Eilenberg-Moore Comparison theorem, Theorem 5.5.11 of \cite{weibel1994introduction}, and as our filtration is exhaustive and complete, it follows that, $H(gr_*\hom^{\mathcal C}(L_0, L_1))\cong H (gr_*\hom^{\mathcal C_c}(L_0, L_1))$. As $c\in\mathbb [0,1]$ was arbitrary, it follows that $\mathcal W^\alpha(M,D)$ and $\mathcal{W}(M,D)$ are connected by a zig-zag of quasi-equivalence. Therefore, $\mathcal{W}^\alpha(M,D)\simeq\mathcal W(M,D)$ as desired.\\
            
            $\textbf{Step 2: }\textit{$\F(M,D)$ is quasi-equivalent to $\mathcal{W}^\alpha(M,D)$}$.\\
            In fact, we will show that we have a strict isomorphism $\F(M,D)\cong\mathcal W^\alpha(M,D)$. To do so, we note that, the $\textit{coefficient}$ moduli space of the $A_\infty$-structure on $\mathcal{W}^\alpha(M,D)$, is in $1$-$1$ correspondence with that of $\F(M,D)$. Indeed, let $\mathbb L=(L_0, \dots, L_k)$ be a Lagrangian label and $y=(y_0, \dots, y_k)$ be a corresponding tuple of Hamiltonian chords. Having Lemma \ref{MaxPrinciple} and Lemma \ref{PositivityIntersection} in mind, the coefficient of $y_0q^d$ in $\hom^{\mathcal W^\alpha(M,D)}(L_0, L_k)$ is given by the rigid count of pseudo-holomorphic polygons in $M$, from a disk $\Sigma\setminus\{z_0,\dots,z_k\}$ with $(k+1)$-boundary punctures and $d$-interior marked points, satisfying $(du-Y)^{0,1}=0$, with asymptotic conditions limiting to $y_i$ at each boundary puncture and limiting to a single point of $D$ at each interior marked point with local intersection number $1$. On the other hand, the coefficient of $y_0q^d$ in $\hom^{\F(M,D)}(L_0, L_k)$ is given by the rigid count of pseudo-holomorphic polygons in $M$, from a disk $\Sigma\setminus\{z_0,\dots,z_k\}$ with $(k+1)$-boundary punctures and $d$-interior marked points, satisfying $(du-Y)^{0,1}=0$, with asymptotic conditions limiting to $y_i$ at each boundary puncture and limiting to a single point of $D$ at each interior marked point with local intersection number $1$. Therefore, the coefficient of $y_0q^d$ in $\hom^{\mathcal W^{\alpha}(M,D)}(L_0,L_1)$ is the same as the coefficient of $y_0q^d$ in $\hom^{\F(M,D)}(L_0, L_1)$. Therefore, under the assumption that we use the same auxiliary data to define $\F(M,D)$ and $\mathcal{W}(M,D)$, it follows that the identity functor $\F(M,D)\rightarrow \mathcal{W}^\alpha(M,D)$ is a well-defined filtered $A_\infty$-functor and hence $\F(M,D)\cong\mathcal W(M,D)$ are strictly isomorphic.\\

            To finish the proof, we combine $\textbf{Step 1}$ and $\textbf{Step 2}$, and the fact that both of $\F(M,D)$ and $\mathcal{W}(M,D)$ are independent of auxiliary choices up to quasi-equivalence. Thus, we have a zig-zag diagram of quasi-equivalences connecting $\mathcal{F}(M,D)$ and $\mathcal{W}(M,D)$, as desired.
        \end{proof}

\appendix
        \section{Gromov's Graph Trick}
        	 Suppose that $(X, J)$ is an almost complex manifold, $L\subset X$ and $(\Sigma, j)$ be a smooth Riemann surface, possibly with boundary. $J$ induces a decomposition of $TX\cong T^{1, 0}X\oplus T^{0, 1}X$ as real vectorbundles over $X$ by the $\pm 1$-eigenspaces of $J$. Denote by $\bigwedge^{0, 1}\Sigma\boxtimes TX\equiv pr_1^*\bigwedge^{0, 1}\Sigma \otimes pr_2^* TX\rightarrow \Sigma\times X $.
		
		\begin{defn}[Perturbed Pseudo-holomorphic Curve]
			Let $\nu\in\Gamma(\bigwedge^{0, 1}\Sigma\boxtimes TX)$ and $u:\Sigma\rightarrow X$ be a smooth map. $u$ is said to be a perturbed $J$-holomorphic curve if 
			\begin{equation}\label{PerturbedHolomorphicCurve}
				\bar{\partial}_{j, J}u+\nu|_{(z, u(z))}=0
			\end{equation}
			\end{defn}
				\begin{lemma}[Graph Trick]\label{GraphTrick}
				Let $\nu\in\Gamma(\bigwedge^{0, 1}\Sigma\boxtimes TX)$ and $u:\Sigma\rightarrow X$ be a smooth map.
				\begin{enumerate}
					\item The $(1,1)$-tensor 
					\[
					J_\nu = \begin{pmatrix}
						j & 0  \\
						2J\nu & J 
					\end{pmatrix}
					\]
					on $\Sigma\times X$ is an almost complex structure. 
					\item Denote by $\tilde{u}(z):=(z, u(z))$ the graph of $u$. Then $u$ is a perturbed $J$-holomorphic curve, that is, a solution of Equation \ref*{PerturbedHolomorphicCurve}, if and only if $\tilde{u}$ is a $J_\nu$-holomorphic curve.
				\end{enumerate}
			\end{lemma}
			
			\begin{proof}
				\begin{enumerate}
					\item Note that $\nu$ is $(j, J)$-linear and hence $J\nu j=\nu$. Now 
					\[
					J_\nu^2= \begin{pmatrix}
						j^2 & 0  \\
						2J\nu j+2J^2\nu & J^2 
					\end{pmatrix}
					\]
					which is $-\id_{T\Sigma\oplus TX}$.
					\item For $z\in\Sigma$, we have \[
					d_z\tilde{u}= \begin{pmatrix}
						1  \\
						d_z u 
					\end{pmatrix}
					\]
					\[
					J_\nu|_{(z, u(z))}\circ d_z\tilde{u}\circ j_z= \begin{pmatrix}
						-1\\
						2J_{(u(z))}\circ \nu_{(z, u(z))}\circ j_z+ J_{(u(z))}\circ d_z u\circ j_z 
					\end{pmatrix}
					\]
				\end{enumerate}
				Thus,
				\[
				\bar{\partial}_{J_\nu}\tilde{u}= \begin{pmatrix}
					0  \\
					\bar{\partial}_{j, J}u+\nu_{(z, u(z))}
				\end{pmatrix}
				\]
				and hence the result.
			\end{proof}

            We will be interested in perturbations that arise from a Hamiltonian or a Hamiltonian vectorfields on graphs of pseudo-holomorphic curves. In what follows, we assume that $(X, \omega, J)$ is a symplectic manifold and $L\subset X$ is a closed Lagrangian.	
            \begin{defn}[Hamiltonian $1$-forms for Graphs]\label{HamiltonianFormsonGraphs}
				A Hamiltonian form is a section $K$ of $\bigwedge^1\Sigma\boxtimes\bigwedge^0 X\cong\bigwedge^1\Sigma\boxtimes\underline{\mathbb R}$ over $\Sigma\times X$, such that $K|_{(T\partial\Sigma)\times L}\equiv 0$.
			\end{defn}
			
			\begin{rmk}
				\begin{enumerate}
					\item By Kunneth decomposition, we have $\bigwedge^1(\Sigma\times X)\cong(\bigwedge^1\Sigma\boxtimes\underline{\mathbb R})\oplus(\underline{\mathbb R}\boxtimes\bigwedge^1 X)$ and hence $K$ is a section of $\bigwedge^1(\Sigma\times X)$ where the second component is constant.
					\item Having the above item in mind and also using the decomposition $\bigwedge^2(\Sigma\times X)\cong(\bigwedge^2\Sigma\boxtimes\underline{\mathbb R})\oplus(\bigwedge^1\Sigma\boxtimes\bigwedge^1X)\oplus(\underline{\mathbb R}\boxtimes\bigwedge^2X)$, we can write the total differential of $K$ as $dK=(d_1K, d_2K, 0)$, where $dK=d_1K-2\operatorname{Alt}(d_2K)$.
				\end{enumerate}
			\end{rmk}

            \begin{defn}[Hamiltonian Vectorfields for Graphs]
				Let $K$ be a Hamiltonian $1$-form on $\Sigma\times X$ and denote by $Y$ the unique section of $\bigwedge^1\Sigma\boxtimes TX\rightarrow\Sigma\times X$ satisfying $pr_2^*\omega(Y, .)=d_2K$.
			\end{defn}
			For $(z, p)\in\Sigma\times X$ we can identify $\bigwedge_z^1\Sigma\otimes T_pX$ by real homomorphisms $T_z\Sigma\rightarrow T_p X$. Using the complex structures $(j, J)$, we define the $(0, 1)$-part of the Hamiltonian perturbation $Y$ by
			\begin{equation*}
				Y^{0, 1}(z, p):=\frac{1}{2}(Y(z, p)+ J_p\circ Y(z, p)\circ j_z).
			\end{equation*}
    	\begin{defn}[Hamiltonian Perturbation for Graphs]
				A Hamiltonian perturbation for graphs is a section of $\bigwedge^{0,1}\Sigma\boxtimes TX$ which is of the form $Y^{0, 1}$ for some Hamiltonian form $K$ on $\Sigma\times X$.
			\end{defn}
			
				We note that the symplectic energy of a Hamiltonian perturbed curve is not a topological quantity. To this extent, we consider the following definition.
			
			\begin{defn}[Curvature of a Hamiltonian Form] \label{CurvatureHamiltonian}
				Let $K$ be a Hamiltonian $1$-form on $\Sigma\times X$. We define its curvature by
				\begin{equation*}
					R(K):= d_1K+pr_2^*\omega(Y, Y)
				\end{equation*}
				which is a section of $\bigwedge^2\Sigma\boxtimes\mathbb R$ over $\Sigma\times X$.
			\end{defn}

            \begin{lemma}[Energy Identity of Hamiltonian Perturbed Curves] \label{EnergyHamiltonianCurve}
				Assume that $(X, \omega)$ is a symplectic manifold and $J$ is an $\omega$-compatible almost complex structure. Let $L\subset X$ be a closed Lagrangian and $K$ a Hamiltonian $1$-form on $\Sigma\times X$. Let $u:(\Sigma, \partial\Sigma)\rightarrow(X, L)$ be such that $(du+Y)^{0,1}=0$. Then, $\frac{1}{2}\int_\Sigma|du+ Y|^2dvol_\Sigma=\int_\Sigma u^*\omega +\int_\Sigma R(K)|_{\tilde{u}}$.
			\end{lemma}
			
			\begin{proof}
				For simplicity assume that $\Sigma$ embeds into $\mathbb{C}$ and use the standard complex coordinates $z=s+it$. Otherwise we argue as follows and use partition of unity to deduce our result. \\
				The result follows by direct computation after noting that
				\begin{equation*}
					u^*\omega=\omega(\partial_su, \partial_tu)ds\wedge dt
				\end{equation*} 
				and that $(du+Y)^{0,1}=0$ can be rewritten as
				\begin{equation*}
					\partial_su=J(Y-\partial_tu)
				\end{equation*}
				and thus
				\begin{equation*}
					\frac{1}{2}|du+Y|^2dvol_\Sigma=\omega(\partial_su+K(\partial_s), \partial_tu+K(\partial_t))= u^*\omega+R(K)|_{\tilde{u}}-\tilde{u}^*dK.
				\end{equation*}
				Now using Stoke's Theorem and the fact that $K\equiv0$ on $\partial\Sigma\times L$, we get the desired result.
			\end{proof}

\bibliographystyle{alpha}
\bibliography{ref}
\end{document}